\documentclass[final]{lms}
\pdfoutput=1
\usepackage{amsmath}
\usepackage{amsfonts}
\usepackage{amssymb}
\usepackage{graphicx}

\newtheorem{theorem}{Theorem}[section] 
\newtheorem{lemma}[theorem]{Lemma}     
\newtheorem{corollary}[theorem]{Corollary}
\newtheorem{proposition}[theorem]{Proposition}

\newnumbered{definition}[theorem]{Definition}
\newnumbered{remark}[theorem]{Remark}
\newnumbered{example}[theorem]{Example}

\title[Construction of transmutation operators]{Construction of transmutation operators and hyperbolic pseudoanalytic functions}

\author{Vladislav V. Kravchenko and Sergii M. Torba}

\journal{}

\classno{30B60, 30G20, 34A25, 35A35, 35C10 (primary), 30B10, 34L16, 34A45, 35L10, 41A50, 47N20, 65N99 (secondary)}

\extraline{Research was supported by CONACYT, Mexico.
Research of second named author was supported by DFFD, Ukraine (GP/F32/030)
and by SNSF, Switzerland (JRP IZ73Z0 of SCOPES 2009--2012).}

\providecommand{\sech}{\mathop{\rm sech}\nolimits}
\begin{document}

\maketitle

\begin{abstract}
A representation for integral kernels of Delsarte transmutation operators is
obtained in the form of a functional series with the exact formulae for the
terms of the series. It is based on the application of hyperbolic
pseudoanalytic function theory and recent results on mapping properties of the
transmutation operators.

The kernel $K_{1}$ of the transmutation operator relating $A=-\frac{d^{2}%
}{dx^{2}}+q_{1}(x)$ and $B=-\frac{d^{2}}{dx^{2}}$ results to be one of the
complex components of a bicomplex-valued hyperbolic pseudoanalytic function
satisfying a Vekua-type hyperbolic equation of a special form. The other
component of the pseudoanalytic function is the kernel of the transmutation
operator relating $C=-\frac{d^{2}}{dx^{2}}+q_{2}(x)$ and $B$ where $q_{2}$ is
obtained from $q_{1}$ by a Darboux transformation. We prove the expansion
theorem and a Runge-type theorem for this special hyperbolic Vekua equation
and using several known results from hyperbolic pseudoanalytic function theory
together with the recently discovered mapping properties of the transmutation
operators obtain the new representation for their kernels. Several examples
are given. Moreover, based on the presented results approaches for numerical
computation of the transmutation kernels and for numerical solution of spectral problems are proposed.
\end{abstract}

\section{Introduction}

The notion of a transmutation operator relating two linear differential
operators was introduced in 1938 by J. Delsarte \cite{Delsarte} and nowadays
represents a widely used tool in the theory of linear differential equations
(see, e.g., \cite{BegehrGilbert}, \cite{Carroll}, \cite{LevitanInverse},
\cite{Marchenko}, \cite{Sitnik}, \cite{Trimeche}). Very often in literature
the transmutation operators are called the transformation operators. Here we
keep the original term introduced by Delsarte and Lions
\cite{DelsarteLions1956}. It is well known that under certain regularity
conditions the transmutation operator transmuting the operator $A=-\frac
{d^{2}}{dx^{2}}+q(x)$ into $B=-\frac{d^{2}}{dx^{2}}$ is a Volterra integral
operator with good properties. Its integral kernel can be obtained as a
solution of a certain Goursat problem for the Klein-Gordon equation with a
variable coefficient. There exist very few examples of the transmutation
kernels available in a closed form (see \cite{KrT2012}).

In the present work we obtain a representation for the kernels of the
transmutation operators for regular Sturm-Liouville operators (with complex
valued coefficients) in the form of a functional series with the exact
formulae for the terms of the series. The result is based on several new
observations. We use our recent result on the construction of the kernel of
the transmutation operator corresponding to a Darboux associated
Schr\"{o}dinger operator \cite{KrT2012} to find out that a bicomplex-valued
function whose one complex component is the transmutation kernel $K_{1}(x,t)$
(for a Schr\"{o}dinger operator $\frac{d^{2}}{dx^{2}}-q_{1}(x)$, $q_{1}\in
C[-b,b]$) and the other complex component is $K_{2}(x,t)$ (for a
Schr\"{o}dinger operator $\frac{d^{2}}{dx^{2}}-q_{2}(x)$, with $q_{2}$
obtained from $q_{1}$ by a Darboux transformation) is a solution of a certain
hyperbolic Vekua equation of a special form (for the theory of elliptic Vekua
equations see \cite{Vekua} as well as \cite{Berskniga} and \cite{APFT}).

In spite of a recent progress reported in \cite{KKT}, \cite{APFT}, \cite{KRT},
\cite{KT2010JMAA} the theory of hyperbolic Vekua equations is considerably
less developed than the theory of classical (elliptic) Vekua equations. For
example, as was shown in \cite{KRT} (see also \cite{APFT}) the construction of
so-called formal powers (solutions of the Vekua equation generalizing the
usual complex powers $(z-z_{0})^{n}$) can be performed using the
definitions\ completely analogous to those introduced by L. Bers
\cite{Berskniga}. Nevertheless no expansion theorem nor a result on the
completeness of the obtained system (a Runge-type theorem) of hyperbolic
formal powers was available. In this work we apply the results from \cite{KRT}
for constructing the formal powers for the hyperbolic Vekua equation arising
in relation with the kernels $K_{1}$ and $K_{2}$ as well as the observation
from \cite{CKT} establishing that they are the result of the transmutation of
the usual powers of the hyperbolic variable. We obtain an expansion theorem
and prove a completeness result. Moreover, we obtain explicit formulae for the
expansion coefficients which leads to the functional series representation for
the kernel $K_{1}$ (and also for $K_{2}$). We give several examples of
application of this result including a numerical computation. Finally, we
propose an alternative method for approximate construction of the
transmutation kernel based on the same hyperbolic formal powers but instead of
the expansion theorem using the obtained completeness result and show that this method leads to an efficient method of solving of Sturm-Liouville spectral problems.

The paper is structured as follows. In Section \ref{SectHPFT} we introduce the hyperbolic and bicomplex numbers, the Vekua equation one of the solutions of which is the bicomplex function $K_{1}-\mathbf{j}K_{2}$. We present several properties of that Vekua equation, construct an infinite system of its solutions called formal powers, explain its relation to generalized  wave polynomials and to spectral parameter power series (SPPS) representation for solutions of Sturm-Liouville equations. We consider a Goursat problem for the hyperbolic Cauchy-Riemann system and obtain certain related results important for what follows. In Section \ref{SectTO} we recall some long known facts about the transmutation operators and give some recent and new results concerning the relation between $K_{1}$ and $K_{2}$, the commutation and the mapping properties of the corresponding transmutation operators. In Section \ref{SectGGTO} we introduce an operator transforming the boundary data of the Goursat problem for the wave equation into the boundary data of the Goursat problem for the operator $\square-q(x)$. Using its properties we obtain a result about a functional series representation for solutions of the equation $\left(  \square-q\right)  u=0$ in terms of generalized wave polynomials. In Section \ref{SectTransmutVekua} we introduce a transmutation operator relating the hyperbolic Cauchy-Riemann operator with the interesting for us hyperbolic Vekua operator. The main observation here is that this transmutation operator possessing important boundedness properties
maps usual powers of the hyperbolic variable $z$ into the formal powers as
well as the usual derivatives with respect to $z$ into the generalized
derivatives in the sense of L. Bers. Here we mainly use the ideas from
\cite{CKM AACA} where similar observations were made in the elliptic setting.
Section \ref{SectETh} is dedicated to the expansion and the Runge-type theorems corresponding to the special Vekua equation under consideration. We show that
a generalized Taylor formula for its solutions is valid under certain
conditions on their Goursat data. The formula involves the formal powers
constructed and studied in preceding sections. The obtained Runge-type theorem
establishes that even when the conditions required for the validity of the
Taylor expansion are not fulfilled any solution of the studied hyperbolic
Vekua equation can be approximated arbitrarily closely by a formal polynomial.
In Section \ref{SectKernelAsSolVekua} we show that $K_{1}$ and $K_{2}$ are necessarily the conjugate components of a solution of the Vekua equation. Using this we give explicit formulae expressing $K_{2}$ in terms of $K_{1}$ and vice versa. Section \ref{SectFIK} contains the main result of this work -- the representation of the transmutation kernel in terms of a generalized Taylor series in formal powers. For the corresponding expansion coefficients we obtain both a recurrent and a direct exact formulae. Several examples are given illustrating their application. In Section \ref{SectApproxKernel} we show how the obtained results can be used for numerical computation of the transmutation kernels. Here together with the approach based on the generalized Taylor expansion theorem we propose a method based on the Runge-type theorem. When sufficiently many derivatives
of the potential taken in the origin are not available this second method is
more convenient. A developed numerical example illustrates its excellent
computational performance. Finally, in Section \ref{SectASL} we show that the developed theory can be applied to efficient numerical solution of Sturm-Liouville spectral problems and allows one to obtain thousands of eigenvalues whose errors are essentially of the same order.

\section{Elements of hyperbolic pseudoanalytic function theory}\label{SectHPFT}

\subsection{Hyperbolic and bicomplex numbers}

By $\mathbf{j}$ we denote the hyperbolic imaginary unit: $\mathbf{j}^{2}=1$
and consider the algebra of hyperbolic numbers, also called duplex numbers
(see, e.g., \cite{Sobczyk})
\[
\mathbb{D}:=\big\{c=a+b\mathbf{j}\ :\ \mathbf{j}^{2}=1,\ a,b\in\mathbb{R}%
\big\}\cong\mathrm{Cl}_{\mathbb{R}}(0,1).
\]
The algebra $\mathbb{D}$ is commutative and contains zero divisors. Additional
information on the hyperbolic numbers can be found in \cite{Lavrentyev and
Shabat}, \cite{MotterRosa} and \cite{Sobczyk}.

We will consider the variable $z=x+t\mathbf{j}$ where $x$ and $t$ are real
variables and the corresponding differential operators
\[
\partial_{z}=\frac{1}{2}\left(  {\partial_{x}+\mathbf{j}\partial_{t}}\right)
\mbox{ and }\partial_{\bar{z}}=\frac{1}{2}\left(  {\partial_{x}-\mathbf{j}%
\partial_{t}}\right)  .
\]
As in the case of complex numbers, we have $\partial_{\bar{z}}z=0$ which
explains the choice of the minus sign in the definition of $\partial_{\bar{z}%
}$.

Let $\mathbb{B}$ denote the algebra of bicomplex numbers which can be defined
as follows
\[
\mathbb{B}:=\big\{w=u+v\mathbf{j}\ :\ u,v\in\mathbb{C}\big\}
\]
and the complex imaginary unit $\mathbf{i}$ commutes with $\mathbf{j}$. More
on bicomplex numbers can be found, e.g., in \cite{CKr2012}, \cite{CastaKr2005}%
, \cite{APFT} and \cite{RochonTrembl}.

The conjugate of a bicomplex number $w$ with respect to $\mathbf{j}$ we denote
by $\overline{w}$, i.e., $\overline{w}=u-v\mathbf{j}$. The corresponding
conjugation operator is denoted by $C$, $Cw=\overline{w}$.

$\mathbb{C}$-valued functions will be also called scalar. For the scalar
components of $w$ we introduce the notations
\[
\mathcal{R}(w)=u=\frac{1}{2}(w+\overline{w})\quad\text{and\quad}%
\mathcal{I}(w)=v=\frac{1}{2\mathbf{j}}(w-\overline{w}).
\]
In order not to overload the text by excessively many different notations we
will use the notation $\mathcal{R}$ and $\mathcal{I}$ also in the operational
sense as projection operators projecting $w\in\mathbb{B}$ onto the respective
scalar components, $\mathcal{R=}\frac{1}{2}(I+C)$ and $\mathcal{I}=\frac
{1}{2\mathbf{j}}(I-C)$ where $I$ is the identity operator.

It is convenient to introduce the pair of idempotents $P^{+}=\frac{1}%
{2}(1+\mathbf{j})$ and $P^{-}=\frac{1}{2}(1-\mathbf{j})$ satisfying $\left(
P^{\pm}\right)  ^{2}=P^{\pm}$ and $P^{+}P^{-}=P^{-}P^{+}=0$. Then for any
$w\in\mathbb{B}$ there exist the unique numbers $w^{+}$, $w^{-}\in\mathbb{C}$
such that $w=P^{+}w^{+}+P^{-}w^{-}$ which are related with the components of
$w$ in the following way
\begin{equation}
w^{\pm}=\mathcal{R}\left(  w\right)  \pm\mathcal{I}(w). \label{W+-}%
\end{equation}
A nonzero element $w\in\mathbb{B}\ $belongs to the set of zero divisors
$\sigma(\mathbb{B)}$ iff $w=P^{+}w^{+}$ or $w=P^{-}w^{-}$.

We will use the following norm (see \cite{CKr2012}) in $\mathbb{B}$,
\begin{equation}
\left\vert w\right\vert =\frac{1}{2}\left(  \left\vert w^{+}\right\vert
_{\mathbb{C}}+\left\vert w^{-}\right\vert _{\mathbb{C}}\right)  ,
\label{bicom_norm}%
\end{equation}
where $\left\vert \cdot\right\vert _{\mathbb{C}}$ is the usual norm in
$\mathbb{C}$.

\subsection{Hyperbolic Vekua equation}

We will consider functions $w:\mathbb{D\rightarrow B}$ and Vekua-type
equations of the form%
\begin{equation}
\partial_{\bar{z}}w=aw+b\overline{w}\label{Vekua hyper}%
\end{equation}
where $a$ and $b$ are bicomplex functions of the variable $z\in\mathbb{D}$. In
\cite{KKT}, \cite{APFT}, \cite{KRT} it was shown that many results from
pseudoanalytic function theory \cite{Berskniga} remain valid in the hyperbolic
situation. We will not give here definitions and properties corresponding to
the general Vekua equation (\ref{Vekua hyper}) referring the reader to
\cite{APFT}. Instead, as we are interested in a very special form of the
coefficients $a$ and $b$ we restrict ourselves to that particular case and
obtain several results, such as the expansion theorem and a Runge-type theorem
which are still unavailable in a general situation. Thus, let us consider the
following hyperbolic Vekua equation
\begin{equation}
\partial_{\bar{z}}W-\frac{\partial_{\bar{z}}f}{f}\overline{W}%
=0\label{Vekua main hyper}%
\end{equation}
where $f$ is a scalar, non-vanishing function in the domain of interest. Vekua
equations of this form, i.e., with $a\equiv0$ and $b$ being a logarithmic
derivative of a scalar function are called Vekua equations of the main type or
main Vekua equations. They arise from the factorization of the stationary
Schr\"{o}dinger operator (in the elliptic case) \cite{Krpseudoan}, \cite{APFT}
and of the variable mass Klein-Gordon operator (in the hyperbolic case)
\cite{KRT}, \cite{APFT}. Moreover, for the purposes of the present paper it is
sufficient to consider the case when $f$ depends only on $x$ and hence
(\ref{Vekua main hyper}) can be written also in the form
\begin{equation}
\partial_{\bar{z}}W-\frac{f^{\prime}(x)}{2f(x)}\overline{W}%
=0.\label{Vekua main x}%
\end{equation}
It is equivalent to the system
\begin{equation}
f{\partial_{x}}\left(  \frac{1}{f}u\right)  ={\partial_{t}v,}\label{Vek1}%
\end{equation}%
\begin{equation}
\frac{1}{f}{\partial_{x}}\left(  fv\right)  ={\partial_{t}u,}\label{Vek2}%
\end{equation}
where $u$, $v$: $\mathbb{R}^{2}\rightarrow\mathbb{C}$ and $W=u+v\mathbf{j}$.
Introducing the notation $\varphi=u/f$ and $\psi=fv$ one can rewrite this
system also in the following form%
\begin{equation}
{\partial_{t}}\varphi=\frac{1}{f^{2}}{\partial_{x}}\psi\quad\text{and\quad
}{\partial_{x}}\varphi=\frac{1}{f^{2}}{\partial_{t}}\psi.\label{hyp p analyt}%
\end{equation}
In the case when all the involved functions were real valued this system was
studied in \cite{PA1983}, \cite{PA1983 2} and later on in \cite{KRT},
\cite{APFT}.

Equation (\ref{Vekua main hyper}) admits a corresponding generating pair $F=f$
and $G=\mathbf{j}/f$. We recall that a generating pair corresponding to a
Vekua equation is a pair of its solutions $F$ and $G$, independent in the
sense that an arbitrary function $\omega$, in this case $\mathbb{B}$-valued,
can be represented in the form
\begin{equation}
\omega=\varphi F+\psi G \label{omega}%
\end{equation}
where $\varphi$ and $\psi$ are scalar functions. That is, $F$ and $G$ should
satisfy the condition $\mathcal{I}(\overline{F}G)\neq0$ everywhere in the
domain of interest. In relation with the generating pair $(F,G)=(f,\mathbf{j}%
/f)$ observe that the scalar functions $\varphi$ and $\psi$ satisfy
(\ref{hyp p analyt}) iff the $\mathbb{B}$-valued function $\omega=\varphi
f+\psi\mathbf{j}/f$ is a solution of (\ref{Vekua main hyper}).

The knowledge of a generating pair allows one to define the $(F,G)$-derivative
in the sense of L. Bers. For a function $\omega$ written in the form
(\ref{omega}) its $(F,G)$-derivative has the form
\[
\overset{\circ}{\omega}:=\frac{d_{(F,G)}\omega}{dz} := \left(  \partial
_{z}\varphi\right)  F+\left(  \partial_{z}\psi\right)  G.
\]
Whenever $W$ satisfies (\ref{Vekua main hyper}) its $(F,G)$-derivative (with
$(F,G)=(f,\mathbf{j}/f)$ ) is a solution of the Vekua equation \cite{APFT}
\begin{equation}
\partial_{\bar{z}}w+\frac{\partial_{z}f}{f}\overline{w}=0.
\label{Vekua successor}%
\end{equation}
Notice that for $f=f(x)$ this equation can be written as follows
\[
\partial_{\bar{z}}w+\frac{f^{\prime}(x)}{2f(x)}\overline{w}=0
\]
or%
\begin{equation}
\partial_{\bar{z}}w-\frac{\left(  f^{-1}(x)\right)  ^{\prime}}{2f^{-1}%
(x)}\overline{w}=0. \label{Vekua succ}%
\end{equation}
Thus, the coefficient in the equation has the same structure as in
(\ref{Vekua main x}) but the role of $f$ is played by $1/f$. Hence a
generating pair for (\ref{Vekua succ}) can be chosen as $(F_{1},G_{1}%
)=(1/f,f\mathbf{j})$. The second derivative of a solution of
(\ref{Vekua main x}) can be defined in the form $\overset{\circ\circ}{W}$
$=\overset{\circ}{w}$ where $w=\overset{\circ}{W}=\varphi_{1}F_{1}+\psi
_{1}G_{1}=\varphi_{1}/f+\psi_{1}f\mathbf{j}$, $\varphi_{1}$ and $\psi_{1}$ are
scalar functions, and $\overset{\circ}{w}=\left(  \partial_{z}\varphi
_{1}\right)  /f+\left(  \partial_{z}\psi_{1}\right)  f\mathbf{j}$. Obviously,
the next generating pair can be chosen again as $(F_{2},G_{2})=(f,\mathbf{j}%
/f)$ $\ $and in this way the derivative of any order in the sense of Bers can
be defined. The derivative of $n$-th order we denote as $W^{[n]}$. In the
theory of pseudoanalytic functions such a sequence of generating pairs is
known as a periodic generating sequence (of the period two). We have
$(F_{n},G_{n})=(f,\mathbf{j}/f)$ when $n$ is even, and $(F_{n},G_{n}%
)=(1/f,f\mathbf{j})$ when $n$ is odd.

For what follows it is important to observe that if $W$ is a solution of
(\ref{Vekua main x}), its $n$-th derivative in the sense of Bers (if exists)
has an especially simple form
\begin{equation}
W^{[n]}=\mathbf{j}^{n}{\partial_{t}^{n}W.} \label{nth Der}%
\end{equation}
Indeed, for $W=\varphi f+\psi\mathbf{j}/f$ we have $\overset{\circ}{W}=\left(
\partial_{z}\varphi\right)  f+\left(  \partial_{z}\psi\right)  \mathbf{j}/f$
where $\varphi$ and $\psi$ are solutions of (\ref{hyp p analyt}). Consider
\[
\partial_{z}\varphi=\frac{1}{2}\left(  {\partial_{x}\varphi+\mathbf{j}%
\partial_{t}}\varphi\right)  =\frac{1}{2}\left(  \frac{1}{f^{2}}{\partial_{t}%
}\psi{+\mathbf{j}\partial_{t}}\varphi\right)  =\frac{{\mathbf{j}}}%
{2f}{\partial_{t}}\left(  \varphi f+\psi\frac{{\mathbf{j}}}{f}\right)
=\frac{{\mathbf{j}}}{2f}{\partial_{t}W}%
\]
and similarly,
\[
\partial_{z}\psi=\frac{1}{2}\left(  {\partial_{x}\psi+\mathbf{j}\partial_{t}%
}\psi\right)  =\frac{1}{2}\left(  f^{2}{\partial_{t}}\varphi{+\mathbf{j}%
\partial_{t}}\psi\right)  =\frac{f}{2}{\partial_{t}}\left(  \varphi
f+\psi\frac{{\mathbf{j}}}{f}\right)  =\frac{f}{2}{\partial_{t}W.}%
\]
Thus, $\overset{\circ}{W}=\mathbf{j}{\partial_{t}W}$. As this reasoning does
not change if $f$ is substituted everywhere by $1/f$ we obtain that
$\overset{\circ\circ}{W}$ $=\mathbf{j}{\partial_{t}}\overset{\circ}{{W}%
}=\mathbf{j}^{2}{\partial_{t}^{2}W}$ and hence (\ref{nth Der}) is valid both
for odd and for even values of $n$.

Below we extend some of the results on the relationship between hyperbolic
pseudoanalytic functions and solutions of the Klein-Gordon equations to the
bicomplex case. Since all the proofs are essentially the same, we refer the
reader to \cite[Chapter 13]{APFT}, \cite{KRT} for details. The operator
$\partial_{\bar{z}}$ applied to a scalar function $\varphi$ can be regarded as
a kind of gradient. If $\partial_{\bar{z}}\varphi=\Phi$, where $\Phi$ is a
$\mathbb{B}$-valued function defined on a simply connected domain $\Omega$
with $\Phi_{1}=\mathcal{R}(\Phi)$ and $\Phi_{2}=\mathcal{I}(\Phi)$ such that
\[
\partial_{t}\Phi_{1}+\partial_{x}\Phi_{2}=0,\quad\forall\,(x,t)\in\Omega,
\]
then we can construct $\varphi$ up to an arbitrary complex constant $c$.
Indeed, let $\Gamma\subset\Omega$ be a rectifiable curve leading from
$(x_{0},t_{0})$ to $(x,t)$, then the integral
\begin{equation}
\overline{A}_{h}[\Phi](x,t):=2\left(  \int_{\Gamma}\Phi_{1}\,dx-\Phi
_{2}\,dt\right)  \label{ABarGamma}%
\end{equation}
is path-independent, and all $\mathbb{C}$-valued solutions $\varphi$ of the
equation $\partial_{\bar{z}}\varphi=\Phi$ in $\Omega$ have the form
$\varphi(x,t)=\overline{A}_{h}[\Phi(x,t)]+c$ where $c$ is an arbitrary complex
constant. In other words the operator $\overline{A}_{h}$ denotes the well
known operation for reconstructing the potential function from its gradient.
In Section \ref{SectKernelAsSolVekua} we need a particular case of
\eqref{ABarGamma} when $\Gamma$ is the path consisting from 2 segments, first
going from $(x_{0},t_{0})$ to $(x,t_{0})$ and second going to $(x,t)$.
Assuming that this path belongs to the domain of interest $\Omega$, in such
case formula \eqref{ABarGamma} reads as follows
\begin{equation}
\overline{A}_{h}[\Phi](x,t)=2\left(  \int_{x_{0}}^{x}\Phi_{1}(\eta
,t_{0})\,d\eta-\int_{t_{0}}^{t}\Phi_{2}(x,\xi)\,d\xi\right)  .
\label{ABarKernel}%
\end{equation}

The following theorems explain relationship between the main Vekua equation
and variable mass Klein-Gordon equations. We use the notation $\Box
:=\partial_{x}^{2}-\partial_{t}^{2}$.

\begin{theorem}
[\cite{KRT}]Let $f=f(x)$ be a non-vanishing scalar function and $W$ be a
solution of \eqref{Vekua main x}. Then $u=\mathcal{R}(W)$ is a solution of the
equation
\begin{equation}
\label{KG1}\bigl(\Box-q_{1}(x)\bigr)u=0,\quad\text{where}\quad q_{1} =
\frac{f^{\prime\prime}}{f},
\end{equation}
and $v = \mathcal{I}(W)$ is a solution of the equation
\begin{equation}
\label{KG2}\bigl(\Box-q_{2}(x)\bigr)v=0,\quad\text{where}\quad q_{2} =
\frac{(f^{-1})^{\prime\prime}}{f^{-1}} = 2\Bigl(\frac{f^{\prime}}{f}%
\Bigr)^{2}-q_{1}.
\end{equation}
\end{theorem}

\begin{theorem}
[\cite{KRT}]\label{ThDarboux} Let $u$ be a scalar solution of the Klein-Gordon
equation \eqref{KG1} in a simply connected domain $\Omega$. Then a scalar
solution $v$ of the associated Klein-Gordon equation \eqref{KG2} such that
$u+\mathbf{j}v$ is a solution of \eqref{Vekua main x} in $\Omega$, can be
constructed according to the formula
\[
v = -f^{-1} \overline{A}_{h}\bigl[\mathbf{j} f^{2} \partial_{\bar{z}}(f^{-1}
u)\bigr].
\]
It is unique up to an additive term $cf^{-1}$ where $c$ is an arbitrary
complex constant.

Vice versa, given a solution $v$ of \eqref{KG2}, the corresponding solution
$u$ of \eqref{KG1} such that $u+\mathbf{j}v$ is a solution of
\eqref{Vekua main x}, has the form
\[
u = -f \overline{A}_{h}\bigl[\mathbf{j} f^{-2} \partial_{\bar{z}}(f v)\bigr],
\]
up to an additive term $cf$.
\end{theorem}

\subsection{Integration and formal powers}

Let $(F,G)$ be a generating pair. Its adjoint generating pair $(F,G)^{\ast
}=(F^{\ast},G^{\ast})$ is defined by the formulae
\[
F^{\ast}=-\frac{2\overline{F}}{F\overline{G}-\overline{F}G},\qquad G^{\ast}=
\frac{2\overline{G}}{F\overline{G}-\overline{F}G}.
\]
For our special case of generating pairs we have $(f, \mathbf{j}/f)^{\ast}=
(f\mathbf{j}, 1/f)$ and $(1/f, f\mathbf{j})^{\ast}= (\mathbf{j}/f, f)$.

The $(F,G)$-integral is defined as
\[
\int_{\Gamma}Wd_{(F,G)}z=\frac{1}{2}\left(  F(z_{1})\,\mathcal{R}\left(
\int_{\Gamma}G^{\ast}W\,dz\right)  +G(z_{1})\,\mathcal{R}\left(  \int_{\Gamma
}F^{\ast}W\,dz\right)  \right)
\]
where $\Gamma$ is a rectifiable curve leading from $z_{0}$ to $z_{1}$.

If $W=\varphi F+\psi G$ is a bicomplex $(F,G)$-pseudoanalytic function where
$\varphi$ and $\psi$ are scalar functions then
\[
\int_{z_{0}}^{z}\overset{\circ}{W}\,d_{(F,G)}z=W(z)-\varphi(z_{0}%
)F(z)-\psi(z_{0})G(z),
\]
and this integral is path-independent and represents the $(F,G)$%
-antiderivative of $\overset{\circ}{W}$.

\begin{definition}
[\cite{Berskniga}]\label{DefFormalPower_bi}The formal power $Z_{m}%
^{(0)}(a,z_{0};z)$ with center at $z_{0}$, coefficient $a$ and exponent $0$ is
defined as the linear combination of the generators $F_{m}$, $G_{m}$ with
scalar constant coefficients $\lambda$, $\mu$ chosen so that $\lambda
F_{m}(z_{0})+\mu G_{m}(z_{0})=a$. The formal powers with exponents
$n=0,1,2,\ldots$ are defined by the recursion formula
\[
Z_{m}^{(n+1)}(a,z_{0};z)=(n+1)\int_{z_{0}}^{z}Z_{m+1}^{(n)}(a,z_{0}%
;\zeta)\,d_{(F_{m},G_{m})}\zeta.
\]
\end{definition}

This definition implies the following properties.

\begin{enumerate}
\item $Z_{m}^{(n)}(a,z_{0};z)$ is an $(F_{m},G_{m})$-bicomplex pseudoanalytic
function of $z$, i.e., it satisfies the hyperbolic Vekua equation determined
by $(F_{m},G_{m})$ (see \cite{APFT}).

\item If $a^{\prime}$ and $a^{\prime\prime}$ are scalar constants, then
\[
Z_{m}^{(n)}(a^{\prime}+\mathbf{j}a^{\prime\prime},z_{0};z)=a^{\prime}%
Z_{m}^{(n)}(1,z_{0};z)+a^{\prime\prime}Z_{m}^{(n)}(\mathbf{j},z_{0};z).
\]

\item The formal powers satisfy the differential relations
\[
\frac{d_{(F_{m},G_{m})}Z_{m}^{(n)}(a,z_{0};z)}{dz}=nZ_{m+1}^{(n-1)}%
(a,z_{0};z).
\]

\item The asymptotic formulae
\[
Z_{m}^{(n)}(a,z_{0};z)\sim a(z-z_{0})^{n},\quad z\rightarrow z_{0}%
\]
hold.
\end{enumerate}

In the following we denote $Z_{0}^{(n)}(a,z_{0};z)$ by $Z^{(n)}(a,z_{0};z)$.

Since we are primarily interested in the hyperbolic Vekua equation
(\ref{Vekua main x}), observe that as was explained above, a corresponding
generating sequence is periodic having the form $(F_{n},G_{n})=(f,\mathbf{j}%
/f)$ when $n$ is even, and $(F_{n},G_{n})=(1/f,f\mathbf{j})$ when $n$ is odd.
Using the construction from \cite{Berskniga} given there for an analogous
elliptic situation (see also \cite[Sect. 4.2]{APFT}) we obtain that the
corresponding formal powers admit the following elegant representation. We
consider the formal powers with the center in the origin and for simplicity
assume that $f(0)=1$. Define the following systems of functions
\begin{align}
X^{(0)}(x) &  \equiv\widetilde{X}^{(0)}(x)\equiv1,\label{X1}\\
X^{(n)}(x) &  =n\int_{0}^{x}X^{(n-1)}(s)\left(  f^{2}(s)\right)  ^{(-1)^{n}%
}\,\mathrm{d}s,\label{X2}\\
\widetilde{X}^{(n)}(x) &  =n\int_{0}^{x}\widetilde{X}^{(n-1)}(s)\left(
f^{2}(s)\right)  ^{(-1)^{n-1}}\,\mathrm{d}s.\label{X3}%
\end{align}
Then the formal powers corresponding to \eqref{Vekua main x} have the
following form. For $\alpha=\alpha^{\prime}+\mathbf{j}\alpha^{\prime\prime}$
we have
\begin{equation}
Z^{(n)}(\alpha,0;z)=f(x)\mathcal{R}\bigl(\vphantom{Z^{(n)}}_{\ast}%
Z^{(n)}(\alpha,0;z)\bigr)+\frac{\mathbf{j}}{f(x)}\mathcal{I}%
\bigl(\vphantom{Z^{(n)}}_{\ast}Z^{(n)}(\alpha,0;z)\bigr),\label{Zn}%
\end{equation}
where
\begin{equation}
_{\ast}Z^{(n)}(\alpha,0;z)=\alpha^{\prime}\sum_{k=0}^{n}\binom{n}{k}%
X^{(n-k)}\mathbf{j}^{k}t^{k}+\mathbf{j}\alpha^{\prime\prime}\sum_{k=0}%
^{n}\binom{n}{k}\widetilde{X}^{(n-k)}\mathbf{j}^{k}t^{k}\qquad\text{for an odd
}n\label{Znodd}%
\end{equation}
and
\begin{equation}
_{\ast}Z^{(n)}(\alpha,0;z)=\alpha^{\prime}\sum_{k=0}^{n}\binom{n}{k}%
\widetilde{X}^{(n-k)}\mathbf{j}^{k}t^{k}+\mathbf{j}\alpha^{\prime\prime}%
\sum_{k=0}^{n}\binom{n}{k}X^{(n-k)}\mathbf{j}^{k}t^{k}\qquad\text{for an even
}n.\label{Zneven}%
\end{equation}

For any $\alpha\in\mathbb{B}$ and $n\in\mathbb{N}_{0}=\mathbb{N}\cup\left\{
0\right\}  $ the formal power (\ref{Zn}) is a solution of (\ref{Vekua main x}).

\begin{remark}
Formulae \eqref{Zn}--\eqref{Zneven} clearly generalize the binomial
representation for the analytic powers $\alpha z^{n}$. If one chooses
$f\equiv1$ then $Z^{(n)}(\alpha,0;z)=\alpha z^{n}$.
\end{remark}

The formal powers $Z^{(n)}(\alpha,0;z)$ can be written with the use of
generalized wave polynomials \cite{KKTT}. Consider two families of functions
$\left\{  \varphi_{k}\right\}  _{k=0}^{\infty}$ and $\left\{  \psi
_{k}\right\}  _{k=0}^{\infty}$ constructed from \eqref{X1}, \eqref{X2} and
\eqref{X3} by
\begin{equation}
\varphi_{k}(x)=%
\begin{cases}
f(x)X^{(k)}(x), & k\text{\ odd},\\
f(x)\widetilde{X}^{(k)}(x), & k\text{\ even},
\end{cases}
\label{phik}%
\end{equation}
and
\begin{equation}
\psi_{k}(x)=%
\begin{cases}
\dfrac{\widetilde{X}^{(k)}(x)}{f(x)}, & k\text{\ odd,}\\
\dfrac{X^{(k)}(x)}{f(x)}, & k\text{\ even}.
\end{cases}
\label{psik}%
\end{equation}
The following functions
\begin{equation}
u_{0}=\varphi_{0}(x),\quad u_{2m-1}(x,t)=\sum_{\text{even }k=0}^{m}\binom
{m}{k}\varphi_{m-k}(x)t^{k},\quad u_{2m}(x,t)=\sum_{\text{odd }k=1}^{m}%
\binom{m}{k}\varphi_{m-k}(x)t^{k}, \label{um}%
\end{equation}
and
\begin{equation}
v_{0}=\psi_{0}(x),\quad v_{2m-1}(x,t)=\sum_{\text{even }k=0}^{m}\binom{m}%
{k}\psi_{m-k}(x)t^{k},\quad v_{2m}(x,t)=\sum_{\text{odd }k=1}^{m}\binom{m}%
{k}\psi_{m-k}(x)t^{k}, \label{vm}%
\end{equation}
are called generalized wave polynomials (the wave polynomials are introduced
below, in Remark \ref{Rem Wave polynomials}). It is easy to see that
\begin{align}
Z^{(0)}(\alpha,0;z)  &  =\alpha^{\prime}u_{0}(x,t)+\mathbf{j}\alpha
^{\prime\prime}v_{0}(x,t), &  & \label{ZasWavePol0}\\
Z^{(n)}(\alpha,0;z)  &  =\alpha^{\prime}u_{2n-1}(x,t)+\alpha^{\prime\prime
}u_{2n}(x,t)+\mathbf{j}\bigl(\alpha^{\prime}v_{2n}(x,t)+\alpha^{\prime\prime
}v_{2n-1}(x,t)\bigr), &  &  n\geq1. \label{ZasWavePol}%
\end{align}
The following parity relations hold for generalized wave polynomials.
\begin{equation}
u_{0}(x,-t)=u_{0}(x,t),\qquad u_{2n-1}(x,-t)=u_{2n-1}(x,t),\qquad
u_{2n}(x,-t)=-u_{2n}(x,t). \label{umParity}%
\end{equation}


It is worth to mention the following result obtained in \cite{KrCV08} (for
additional details and simpler proof see \cite{APFT} and \cite{KrPorter2010})
which establishes the relation of the systems of functions $\left\{
\varphi_{k}\right\}  _{k=0}^{\infty}$ and $\left\{  \psi_{k}\right\}
_{k=0}^{\infty}$ to the Sturm-Liouville equation.

\begin{theorem}
[\cite{KrCV08}]\label{ThGenSolSturmLiouville} Let $q$ be a continuous complex
valued function of an independent real variable $x\in\lbrack a,b]$ and
$\lambda$ be an arbitrary complex number. Suppose there exists a solution $f$
of the equation
\begin{equation}
f^{\prime\prime}-qf=0 \label{SLhom}%
\end{equation}
on $(a,b)$ such that $f\in C^{2}(a,b)\cap C^{1}[a,b]$ and $f(x)\neq0$\ for any
$x\in\lbrack a,b]$. Then the general solution $g\in C^{2}(a,b)\cap C^{1}[a,b]$
of the equation
\[
g^{\prime\prime}-qg=\lambda g
\]
on $(a,b)$ has the form
\[
g=c_{1}g_{1}+c_{2}g_{2}%
\]
where $c_{1}$ and $c_{2}$ are arbitrary complex constants,
\begin{equation}
g_{1}=\sum_{k=0}^{\infty}\frac{\lambda^{k}}{(2k)!}\varphi_{2k}\qquad
\text{and}\qquad g_{2}=\sum_{k=0}^{\infty}\frac{\lambda^{k}}{(2k+1)!}%
\varphi_{2k+1} \label{u1u2}%
\end{equation}
and both series converge uniformly on $[a,b]$ together with the series of the
first derivatives which have the form
\begin{multline}
\displaybreak[2]g_{1}^{\prime}=f^{\prime}+\sum_{k=1}^{\infty}\frac{\lambda
^{k}}{(2k)!}\left(  \frac{f^{\prime}}{f}\varphi_{2k}+2k\,\psi_{2k-1}\right)
\qquad\text{and}\label{du1du2}\\
g_{2}^{\prime}=\sum_{k=0}^{\infty}\frac{\lambda^{k}}{(2k+1)!}\left(
\frac{f^{\prime}}{f}\varphi_{2k+1}+\left(  2k+1\right)  \psi_{2k}\right)  .
\end{multline}
The series of the second derivatives converge uniformly on any
segment\linebreak$\lbrack a_{1},b_{1}]\subset(a,b)$.
\end{theorem}

Representations \eqref{u1u2} and \eqref{du1du2}, also known as the SPPS
method, present an efficient and highly competitive technique for solving a
variety of spectral and scattering problems related to Sturm-Liouville
equations, see \cite{CKOR}, \cite{ErbeMertPeterson2012}, \cite{KKRosu},
\cite{KiraRosu2010}, \cite{KrPorter2010} and \cite{KT Obzor} and references therein.

\subsection{The Goursat problem for the hyperbolic Cauchy-Riemann system and a
Runge type theorem}

For the construction of the kernels of the transmutation operators we are
interested in solving corresponding Goursat problems. Our results concerning
the Goursat problems for equation (\ref{Vekua main x}) arising as a natural
step in that construction are based on the following simplest case.

Consider the hyperbolic Cauchy-Riemann system%
\begin{equation}
\partial_{\bar{z}}w=0\quad\Longleftrightarrow\quad%
\begin{cases}
u_{x}=v_{t}\\
u_{t}=v_{x}%
\end{cases}
\label{hCR}%
\end{equation}
studied in a considerable number of publications (see, e.g., \cite{Lavrentyev
and Shabat}, \cite{MotterRosa}, \cite{Wen}). It is easy to see that its
general solution can be written as follows%
\begin{equation}
\label{gensolhCR}%
\begin{split}
w(x,t)  &  =\frac{1}{2}\left(  \Phi\left(  \frac{x+t}{2}\right)  +\Psi\left(
\frac{x-t}{2}\right)  +\mathbf{j}\left(  \Phi\left(  \frac{x+t}{2}\right)
-\Psi\left(  \frac{x-t}{2}\right)  \right)  \right) \\
&  =P^{+}\Phi\left(  \frac{x+t}{2}\right)  +P^{-}\Psi\left(  \frac{x-t}%
{2}\right)
\end{split}
\end{equation}
where $\Phi$ and $\Psi$ are arbitrary continuously differentiable scalar functions.

Considering the values of $w$ on the characteristics $x=t$ and $x=-t$ we see
that
\begin{equation}
w(x,x)=P^{+}\Phi\left(  x\right)  +P^{-}\Psi\left(  0\right)  \quad
\text{and\quad}w(x,-x)=P^{+}\Phi\left(  0\right)  +P^{-}\Psi\left(  x\right)
. \label{condG}%
\end{equation}
Thus, we arrive at the following correct statement of the Goursat problem.
\emph{Find a solution }$w$\emph{ of \eqref{hCR} in a closed square }%
$\overline{R}$\emph{ with the vertices }$(\pm2b,0)$\emph{ and }$(0,\pm
2b)$\emph{, }$b>0$\emph{ satisfying the conditions \eqref{condG} where }$\Phi
$\emph{ and }$\Psi$\emph{ are arbitrary scalar functions such that }$\Phi
,\Psi\in C^{1}[-b,b]$\emph{. }

\begin{proposition}
\label{Prop Goursat hCR} The unique solution of the Goursat problem
\eqref{hCR}, \eqref{condG} has the form \eqref{gensolhCR}.
\end{proposition}

\begin{proof}
The fact that $w$ is a solution of (\ref{hCR}) follows from (\ref{gensolhCR}).
The conditions (\ref{condG}) are obviously satisfied. To prove the uniqueness
one considers the Goursat problem with the homogeneous conditions and finds
out that $\Phi$ and $\Psi$ from the representation (\ref{gensolhCR}) of the
general solution of (\ref{hCR}) satisfy the equations $P^{+}\Phi(x)+P^{-}%
\Psi(0)\equiv0$ and $P^{+}\Phi(0)+P^{-}\Psi(x)\equiv0$, $x\in\lbrack-b,b]$
from which both $\Phi$ and $\Psi$ are necessarily trivial.
\end{proof}

This proposition is not entirely new. As a kind of analogue of the Cauchy
integral formula for analytic functions it was discussed, e.g., in
\cite{MotterRosa} in the case $w:$ $\mathbb{D}\rightarrow\mathbb{D}$. For our
purposes it is more convenient to consider this simple representation of any
hyperbolic analytic function via its values on the characteristics from the
point of view of the Goursat problem.

The following result is a direct corollary of Proposition
\ref{Prop Goursat hCR}.

\begin{proposition}
\label{Prop Goursat polynom} Let $\Phi$ and $\Psi$ be scalar functions defined
on $[-b,b]$ and admitting there the uniformly convergent power series
expansions%
\[
\Phi\left(  x\right)  =\sum\limits_{n=0}^{\infty}\alpha_{n}x^{n}%
\quad\text{and\quad}\Psi\left(  x\right)  =\sum\limits_{n=0}^{\infty}\beta
_{n}x^{n}.
\]
Then the unique solution of the Goursat problem \eqref{hCR}, \eqref{condG} has
the form
\begin{equation}
w=\sum\limits_{n=0}^{\infty}a_{n}z^{n}, \label{power series}%
\end{equation}
where
\begin{equation}
\label{power series coeffs}a_{n}=\frac{1}{2^{n}}\left(  P^{+}\alpha_{n}%
+P^{-}\beta_{n}\right)  .
\end{equation}
\end{proposition}

\begin{proof}
Due to Proposition \ref{Prop Goursat hCR} we have
\begin{align*}
w(x,t)  &  =P^{+}\sum\limits_{n=0}^{\infty}\alpha_{n}\left(  \frac{x+t}%
{2}\right)  ^{n}+P^{-}\sum\limits_{n=0}^{\infty}\beta_{n}\left(  \frac{x-t}%
{2}\right)  ^{n}\\
&  =\sum\limits_{n=0}^{\infty}\left(  \frac{\alpha_{n}}{2^{n}}P^{+}\left(
x+t\right)  ^{n}+\frac{\beta_{n}}{2^{n}}P^{-}\left(  x-t\right)  ^{n}\right)
.
\end{align*}
Notice that $P^{\pm}z=P^{\pm}\left(  x\pm t\right)  $ and hence $P^{\pm}%
z^{n}=P^{\pm}\left(  x\pm t\right)  ^{n}$. Thus,
\[
w=\sum\limits_{n=0}^{\infty}\left(  \frac{\alpha_{n}}{2^{n}}P^{+}z^{n}
+\frac{\beta_{n}}{2^{n}}P^{-}z^{n}\right)  =\sum\limits_{n=0}^{\infty}\frac
{1}{2^{n}}\left(  P^{+}\alpha_{n}+P^{-}\beta_{n}\right)  z^{n}.
\]
\end{proof}

\begin{remark}
Under the conditions of Proposition \ref{Prop Goursat polynom} we have that
the unique solution $w$ of the Goursat problem \eqref{hCR}, \eqref{condG} can
be represented as a uniformly convergent in $\overline{R}$ Taylor series
\eqref{power series} where
\begin{equation}
a_{n}=\frac{w^{[n]}(0)}{n!}. \label{Taylor coefs}%
\end{equation}
Indeed, let us prove this formula for the coefficients. Due to \eqref{nth Der}
we have that $w^{[n]}=\mathbf{j}^{n}{\partial_{t}^{n}w}$ and due to
\eqref{gensolhCR},
\[
\mathbf{j}^{n}{\partial_{t}^{n}w=}\frac{1}{2^{n}}\left(  P^{+}\mathbf{j}%
^{n}\Phi^{(n)}\left(  \frac{x+t}{2}\right)  +P^{-}\mathbf{j}^{n}\left(
-1\right)  ^{n}\Psi^{(n)}\left(  \frac{x-t}{2}\right)  \right)  .
\]
Now we notice that $P^{+}\mathbf{j}^{n}=P^{+}$ and $P^{-}\mathbf{j}^{n}\left(
-1\right)  ^{n}=P^{-}$ for any $n=0,1,\ldots$. Thus, $w^{[n]}(0)=\frac
{n!}{2^{n}}\left(  P^{+}\alpha_{n}+P^{-}\beta_{n}\right)  =n!a_{n}$.
\end{remark}

\begin{remark}
\label{Rem Wave polynomials}It is possible to give another representation for
the solution of the Goursat problem, namely, a representation as a series in
wave polynomials, cf. \cite[Proposition 1]{KKTT}. Recall that the following
polynomials
\[
p_{0}(x,t)=1,\quad p_{2m-1}(x,t)=\mathcal{R}\bigl((x+\mathbf{j}t)^{m}%
\bigr),\quad p_{2m}(x,t)=\mathcal{I}\bigl((x+\mathbf{j}t)^{m}\bigr),\ m\geq1,
\]
are called wave polynomials. They may be also written as
\begin{equation}
p_{0}(x,t)=1,\quad p_{2m-1}(x,t)=\sum_{\mathrm{even}\text{ }k=0}^{m}\binom
{m}{k}x^{m-k}t^{k},\quad p_{2m}(x,t)=\sum_{\mathrm{odd}\text{ }k=1}^{m}%
\binom{m}{k}x^{m-k}t^{k}, \label{WavePolynomials}%
\end{equation}
cf.\ with the definition of generalized wave polynomials \eqref{um},
\eqref{vm}. Since $z^{n}=p_{2n-1}(x,t)+\mathbf{j}p_{2n}(x,t)$ for $n\geq1$, we
obtain from \eqref{power series coeffs} that
\begin{multline*}
w(x,t)=(P^{+}\alpha_{0}+P^{-}\beta_{0})p_{0}(x,t)\\
+\sum_{n=1}^{\infty}
\frac{P^{+}\alpha_{n}+P^{-}\beta_{n}}{2^{n}}p_{2n-1}(x,t)+\sum_{n=1}^{\infty
}\frac{P^{+}\alpha_{n}-P^{-}\beta_{n}}{2^{n}}p_{2n}(x,t).
\end{multline*}
\end{remark}

Obviously, an arbitrary continuously differentiable solution of (\ref{hCR})
cannot be represented globally as a uniformly convergent power series of the
form (\ref{power series}). Nevertheless the following analogue of the Runge
approximation theorem from classical complex analysis is a corollary of the
preceding proposition.

\begin{proposition}
\label{Prop RungeHCR} Let $w\in C^{1}(\overline{R})$ be a solution of
\eqref{hCR} in $R$. Then there exists a sequence of polynomials $P_{N}%
=\sum\limits_{n=0}^{N}a_{n}z^{n}$ uniformly convergent to $w$ in $\overline
{R}$.
\end{proposition}

\begin{proof}
We need to prove that for any $\varepsilon>0$ there exist such a number $N$
and coefficients $a_{n}$, $n=0,1,\ldots,N$ that $\left\vert w(x,t)-P_{N}%
(x,t)\right\vert <\varepsilon$ for any point $(x,t)\in\overline{R}$. Let the
scalar functions $\Phi$ and $\Psi$ defined on $[-b,b]$ be such that
(\ref{condG}) hold. As was shown above such scalar functions always exist and
under the condition of the proposition they are obviously from $C^{1}[-b,b]$.
We choose $\varepsilon>0$ and such $\varepsilon_{1,2}>0$ that $\varepsilon
=(\varepsilon_{1}+\varepsilon_{2})/2$. According to the Weierstrass theorem
there exists such a number $N$ and such polynomials $p_{1}$ and $p_{2}$ of
order not greater than a certain $N$ that $\left\vert \Phi\left(  x\right)
-p_{1}(x)\right\vert _{\mathbb{C}}<\varepsilon_{1}$ and $\left\vert
\Psi\left(  x\right)  -p_{2}(x)\right\vert _{\mathbb{C}}<\varepsilon_{2}$,
$-b\leq x\leq b$. Due to Proposition \ref{Prop Goursat polynom} the unique
solution $\widetilde{w}$ of the Goursat problem for equation (\ref{hCR}) with
the boundary conditions
\[
\widetilde{w}(x,x)=P^{+}p_{1}\left(  x\right)  +P^{-}p_{2}\left(  0\right)
,\qquad\widetilde{w}(x,-x)=P^{+}p_{1}\left(  0\right)  +P^{-}p_{2}\left(
x\right)
\]
has the form $\widetilde{w}=P_{N}$ where $P_{N}=\sum_{n=0}^{N}\frac{1}{2^{n}%
}\left(  P^{+}\alpha_{n}+P^{-}\beta_{n}\right)  z^{n}$ with $\alpha_{n}$ and
$\beta_{n}$ being the coefficients of the polynomials $p_{1}$ and $p_{2}$ respectively.

Consider
\[%
\begin{split}
\left\vert w(x,t)-\widetilde{w}(x,t)\right\vert  &  =\left\vert w(x,t)-P_{N}
(x,t)\right\vert \\
&  =\left\vert P^{+}\Phi\left(  \frac{x+t}{2}\right)  +P^{-}\Psi\left(
\frac{x-t}{2}\right)  -P^{+}P_{N}^{+}(x,t)-P^{-}P_{N}^{-}(x,t)\right\vert \\
&  =\frac{1}{2}\left(  \left\vert \Phi\left(  \frac{x+t}{2}\right)  -P_{N}
^{+}(x,t)\right\vert _{\mathbb{C}}+\left\vert \Psi\left(  \frac{x-t}
{2}\right)  -P_{N}^{-}(x,t)\right\vert _{\mathbb{C}}\right)  ,
\end{split}
\]
where we have used \eqref{W+-} and \eqref{bicom_norm}. Notice that $P_{N}%
^{+}(x,t)=\sum_{n=0}^{N}\alpha_{n}\left(  \frac{x+t}{2}\right)  ^{n}%
=p_{1}\left(  \frac{x+t}{2}\right)  $ and $P_{N}^{-}(x,t)=\sum_{n=0}^{N}%
\beta_{n}\left(  \frac{x-t}{2}\right)  ^{n}=p_{2}\left(  \frac{x-t}{2}\right)
$. Thus, $\left\vert w(x,t)-P_{N}(x,t)\right\vert <\frac{1}{2}(\varepsilon
_{1}+\varepsilon_{2})=\varepsilon$.
\end{proof}

\section{Transmutation operators}\label{SectTO}

We give a definition of a transmutation operator from \cite{KT Obzor} which is
a modification of the definition given by Levitan \cite{LevitanInverse}
adapted to the purposes of the present work. Let $E$ be a linear topological
space and $E_{1}$ its linear subspace (not necessarily closed). Let $A$ and
$B$ be linear operators: $E_{1}\rightarrow E$.

\begin{definition}
\label{DefTransmut} A linear invertible operator $T$ defined on the whole $E$
such that $E_{1}$ is invariant under the action of $T$ is called a
transmutation operator for the pair of operators $A$ and $B$ if it fulfills
the following two conditions.

\begin{enumerate}
\item Both the operator $T$ and its inverse $T^{-1}$ are continuous in $E$;

\item The following operator equality is valid
\begin{equation}
AT=TB \label{ATTB}%
\end{equation}
or which is the same
\[
A=TBT^{-1}.
\]
\end{enumerate}
\end{definition}

Our main interest concerns the situation when $A=-\frac{d^{2}}{dx^{2}}+q(x)$,
$B=-\frac{d^{2}}{dx^{2}}$, \ and $q$ is a continuous complex-valued function.
Hence for our purposes it will be sufficient to consider the functional space
$E=C[a,b]$ with the topology of uniform convergence and its subspace $E_{1}$
consisting of functions from $C^{2}\left[  a,b\right]  $. One of the
possibilities to introduce a transmutation operator on $E$ was considered by
Lions \cite{Lions57} and later on in other references (see, e.g.,
\cite{Marchenko}), and consists in constructing a Volterra integral operator
corresponding to a midpoint of the segment of interest. As we begin with this
transmutation operator it is convenient to consider a symmetric segment
$[-b,b]$ and hence the functional space $E=C[-b,b]$. It is worth mentioning
that other well known ways to construct the transmutation operators (see,
e.g., \cite{LevitanInverse}, \cite{Trimeche}) imply imposing initial
conditions on the functions and consequently lead to transmutation operators
satisfying (\ref{ATTB}) only on subclasses of $E_{1}$.

Thus, consider the space $E=C[-b,b]$. In \cite{CKT} and \cite{KrT2012} a
parametrized family of transmutation operators for the defined above $A$ and
$B$ was studied. Operators of this family can be realized in the form of the
Volterra integral operator
\begin{equation}
\mathbf{T}_{h} u(x)=u(x)+\int_{-x}^{x}\mathbf{K}(x,t;h)u(t)dt \label{Tmain}%
\end{equation}
where $\mathbf{K}(x,t;h)=\mathbf{H}\big(\frac{x+t}{2},\frac{x-t}{2};h\big)$,
$h$ is a complex parameter, $|t|\le|x|\le b$ and $\mathbf{H}$ is the unique
solution of the Goursat problem
\begin{equation}
\frac{\partial^{2}\mathbf{H}(u,v;h)}{\partial u\,\partial v}=q(u+v)\mathbf{H}%
(u,v;h), \label{GoursatTh1}%
\end{equation}%
\begin{equation}
\mathbf{H}(u,0;h)=\frac{h}{2}+\frac{1}{2}\int_{0}^{u}q(s)\,ds,\qquad
\mathbf{H}(0,v;h)=\frac{h}{2}. \label{GoursatTh2}%
\end{equation}
If the potential $q$ is continuously differentiable, the kernel $\mathbf{K}$
itself is a solution of the Goursat problem
\begin{equation}
\label{GoursatKh1}\left(  \frac{\partial^{2}}{\partial x^{2}}-q(x)\right)
\mathbf{K}(x,t;h)=\frac{\partial^{2}}{\partial t^{2}}\mathbf{K}(x,t;h),
\end{equation}
\begin{equation}
\label{GoursatKh2}\mathbf{K}(x,x;h)=\frac{h}{2}+\frac{1}{2}\int_{0}%
^{x}q(s)\,ds,\qquad\mathbf{K}(x,-x;h)=\frac{h}{2}.
\end{equation}
If the potential $q$ is $n$ times continuously differentiable, the kernel
$\mathbf{K}(x,t;h)$ is $n+1$ times continuously differentiable with respect to
both independent variables.

\begin{remark}
In the case $h=0$ the operator $\mathbf{T}_{h}$ coincides with the
transmutation operator studied in \cite[Chap. 1, Sect. 2]{Marchenko}. In
\cite{LevitanInverse}, \cite{Lions57}, \cite{Trimeche} it was established that
in the case $q\in C^{1}[-b,b]$ the Volterra-type integral operator
\eqref{Tmain} is a transmutation in the sense of Definition \ref{DefTransmut}
on the space $C^{2}[-b,b]$ if and only if the integral kernel $\mathbf{K}%
(x,t)$ satisfies the Goursat problem \eqref{GoursatKh1}, \eqref{GoursatKh2}.
\end{remark}

The following proposition shows that it is sufficient to know the
transmutation operator $\mathbf{T}_{h_{1}}$ or its integral kernel
$\mathbf{K}_{h_{1}}$ for some particular parameter $h_{1}$ to be able to
construct transmutation operators $\mathbf{T}_{h}$ or their integral kernels
$\mathbf{K}_{h}$ for arbitrary values of the parameter $h$.

\begin{proposition}
[\cite{CKT}, \cite{KT Obzor}]\label{PropChangeOfH} The operators
$\mathbf{T}_{h_{1}}$ and $\mathbf{T}_{h_{2}}$ are related by the expression
\[
\mathbf{T}_{h_{2}}u=\mathbf{T}_{h_{1}}\bigg[u(x)+\frac{h_{2}-h_{1}}2\int
_{-x}^{x} u(t)\,dt\bigg]
\]
valid for any $u\in C[-b,b]$.

The integral kernels $\mathbf{K}(x,t;h_{1})$ and $\mathbf{K}(x,t;h_{2})$ are
related by the expression
\[
\mathbf{K}(x,t;h_{2})=\frac{h_{2}-h_{1}}{2}+\mathbf{K}(x,t;h_{1})+\frac
{h_{2}-h_{1}}{2}\int_{t}^{x}\big( \mathbf{K}(x,s;h_{1})-\mathbf{K}%
(x,-s;h_{1})\big) \,ds.
\]
\end{proposition}

The following theorem states that the operators $\mathbf{T}_{h}$ are indeed
transmutations in the sense of Definition \ref{DefTransmut}.

\begin{theorem}
\label{Th Transmutation} Let $q\in C[-b,b]$. Then the operator $\mathbf{T}%
_{h}$ given by \eqref{Tmain} satisfies the equality
\begin{equation}
\left(  -\frac{d^{2}}{dx^{2}}+q(x)\right)  \mathbf{T}_{h}[u]=\mathbf{T}%
_{h}\left[  -\frac{d^{2}}{dx^{2}}(u)\right]  \label{ThTransm}%
\end{equation}
for any $u\in C^{2}[-b,b]$.
\end{theorem}

\begin{remark}
This theorem under additional assumptions on $q$ was proved in \cite{CKT} and
\cite{KrT2012}. In \cite{KT Obzor} it was shown that \eqref{ThTransm} holds
for any $h$ whenever it holds for some particular value $h_{1}$.
\end{remark}

\begin{proof}
It follows from Proposition \ref{PropChangeOfH} that it is sufficient to prove
\eqref{ThTransm} for one particular $h$, see \cite[Proof of Theorem 5.6]{KT
Obzor} for details. We take $h=0$ and the operator $\mathbf{T}:=\mathbf{T}%
_{0}$. Let $\mathbf{K}$ denotes the integral kernel of $\mathbf{T}$. There
exists a sequence of potentials $\{q_{n}\}_{n\in\mathbb{N}}\subset
C^{1}[-b,b]$ such that $q_{n}\to q$, $n\to\infty$ uniformly on $[-b,b]$. Let
$\mathbf{K}_{q_{n}}$ and $\mathbf{T}_{q_{n}}$ be the integral kernels and the
transmutation operators corresponding to the potentials $q_{n}$. Then
$\mathbf{K}_{q_{n}}\to\mathbf{K}$ uniformly and $\mathbf{T}_{n}\to\mathbf{T}$
in the operator norm, see, e.g., \cite[Chap. 1, Sect. 2]{Marchenko}. Take a
function $u\in C^{2}[-b,b]$. For the operators $\mathbf{T}_{q_{n}}$ the
following equality holds \cite[Theorem 6]{KrT2012}
\[
\left(  -\frac{d^{2}}{dx^{2}}+q_{n}(x)\right)  \mathbf{T}_{q_{n}%
}[u]=\mathbf{T}_{q_{n}}\left[  -\frac{d^{2}}{dx^{2}}(u)\right]  ,
\]
which with the use of \eqref{Tmain} may be rewritten as
\[
\frac{d^{2}}{dx^{2}}\biggl(\int_{-x}^{x} \mathbf{K}_{q_{n}}%
(x,t)u(t)\,dt\biggr) = -u^{\prime\prime}(x) + q_{n}(x) \mathbf{T}_{q_{n}%
}[u](x) + \mathbf{T}_{q_{n}}[u^{\prime\prime}](x).
\]
Let $y_{n}(x) := \int_{-x}^{x} \mathbf{K}_{q_{n}}(x,t)u(t)\,dt$ and $z_{n} :=
-u^{\prime\prime}+q_{n} \mathbf{T}_{q_{n}}[u] + \mathbf{T}_{q_{n}}%
[u^{\prime\prime}]$. Then $y_{n}\to y := \int_{-x}^{x} \mathbf{K}%
(x,t)u(t)\,dt$ and $z_{n} \to z:= -u^{\prime\prime}+ q \mathbf{T}[u] +
\mathbf{T}[u^{\prime\prime}]$, $n\to\infty$. Moreover, functions $y_{n}$
satisfy the initial conditions $y_{n}(0)=y^{\prime}_{n}(0)=0$. Now
\eqref{ThTransm} follows from the fact that the operator $\partial_{x}^{2}$
with the domain $\{ y\in C^{2}[-b,b] : y(0)=y^{\prime}(0)=0\}$ is closed, see,
e.g., \cite{Kato}.
\end{proof}

The following theorem summarizes the mapping properties of the operators
$\mathbf{T}_{h}$ establishing that there exists a value of the parameter $h$
such that $\mathbf{T}_{h}$ maps $x^{k}$ to $\varphi_{k}$ defined by
(\ref{phik}).

\begin{theorem}
[\cite{CKT}, \cite{KrT2012}]\label{Th Transmute}Let $q$ be a continuous
complex valued function of an independent real variable $x\in\lbrack-b,b]$ for
which there exists a particular solution $f$ of \eqref{SLhom} such that $f\in
C^{2}[-b,b]$, $f\neq0$ on $[-b,b]$ and normalized as $f(0)=1$. Denote
$h:=f^{\prime}(0)\in\mathbb{C}$. Suppose $\mathbf{T}_{h}$ is the operator
defined by \eqref{Tmain} and $\varphi_{k}$, $k\in\mathbb{N}_{0}$ are functions
defined by \eqref{phik}. Then
\begin{equation}
\label{mapping powers 1}\mathbf{T}_{h} x^{k} = \varphi_{k}(x) \qquad\text{for
any}\ k\in\mathbb{N}_{0}.
\end{equation}
Moreover, $\mathbf{T}_{h}$ maps a solution $v$ of an equation $v^{\prime
\prime}+\omega^{2}v=0$, where $\omega$ is a complex number, into a solution
$u$ of the equation $u^{\prime\prime}-q(x)u+\omega^{2}u=0$ with the following
correspondence of the initial values
\begin{equation}\label{mapping IC}
u(0)=v(0),\qquad u^{\prime}(0)=v^{\prime}(0)+hv(0).
\end{equation}
\end{theorem}

\begin{remark}
The mapping property \eqref{mapping powers 1} of the transmutation operator
allows one to see that the SPPS representations \eqref{u1u2} from Theorem
\ref{ThGenSolSturmLiouville} are nothing but the images of Taylor expansions
of the functions $\cosh\sqrt{\lambda}x$ and $\frac{1}{\sqrt{\lambda}}
\sinh\sqrt{\lambda}x$. Moreover, equality \eqref{mapping powers 1} is behind a
new method for solving Sturm-Liouville problems proposed in \cite{KT MMET
2012}.
\end{remark}

In what follows we assume that $f\in C^{2}[-b,b]$, $f\neq0$ on $[-b,b]$,
$f(0)=1$ and denote $h:=f^{\prime}(0)\in\mathbb{C}$. Any such function is
associated with an operator $\mathbf{T}_{h}$. For convenience, from now on we
will write $T_{f}$ instead of $\mathbf{T}_{h}$ and the integral kernel of
$T_{f}$ will be denoted by $\mathbf{K}_{f}$. Together with the function $f$
let us consider the function $1/f$. Notice that if $f$ is a solution of
(\ref{SLhom}) with the potential $q_{f}=f^{\prime\prime}/f$, the function
$1/f$ is a solution of the equation
\[
\left(  -\frac{d^{2}}{dx^{2}}+q_{1/f}(x)\right)  u=0
\]
with the Darboux associated potential
\[
q_{1/f}=-q_{f}+2\Bigl(\frac{f^{\prime}}{f}\Bigr)^{2}.
\]
Consider the transmutation operator $T_{1/f}$ which satisfies the equality
\[
\left(  -\frac{d^{2}}{dx^{2}}+q_{1/f}(x)\right)  {T}_{1/f}[u]=T_{1/f}\left[  -\frac{d^{2}}{dx^{2}}(u)\right]
\]
for any $u\in C^{2}[-b,b]$ and transforms $x^{k}$ into the functions $\psi
_{k}(x)$, $k\in\mathbb{N}_{0}$ defined by (\ref{psik}), i.e.
\begin{equation}
{T}_{1/f}x^{k}=\psi_{k}(x)\qquad\text{for any}\ k\in\mathbb{N}_{0}.
\label{mapping powers 2}%
\end{equation}

It follows from \eqref{mapping powers 1}, \eqref{mapping powers 2} and
definitions \eqref{um}, \eqref{vm} and \eqref{WavePolynomials} that
generalized wave polynomials are images of wave polynomials under the action
of operators $T_{f}$ and $T_{1/f}$, i.e.
\begin{equation}
\label{Mapping Wave Polynomial}T_{f} p_{n} = u_{n} \quad\text{and}\quad
T_{1/f} p_{n} = v_{n}\quad\text{for any }n\ge0.
\end{equation}

In \cite{KrT2012} explicit formulas were obtained for the kernel
$\mathbf{K}_{1/f}(x,t)$ in terms of $\mathbf{K}_{f}(x,t)$. In order to obtain
a simpler expression for the integral kernel $\mathbf{K}_{1/f}(x,t)$ we
assumed that the original integral kernel $\mathbf{K}_{f}(x,t)$ is known in
the larger domain than required by definition \eqref{Tmain}. Namely, suppose
that the function $\mathbf{K}_{f}(x,t)$ is known and is continuously
differentiable in the domain $\bar{\Pi}:\ -b\leq x\leq b,-b\leq t\leq b$. It
is worth mentioning that such continuation of the integral kernel to the
domain $\bar{\Pi}$ is always possible due to \eqref{GoursatTh1},
\eqref{GoursatTh2} and the general theory of Goursat problems. We refer the
reader to \cite{KrT2012} for further details. In Corollary \ref{Cor K Vekua}
based on pseudoanalytic function theory we obtain another expression for the
integral kernel $\mathbf{K}_{1/f}(x,t)$ suitable in the case when the integral
kernel $\mathbf{K}_{f}(x,t)$ is known only in the natural domain $|t|\le|x|\le
b$.

\begin{proposition}
[\cite{KrT2012}]\label{PropTT_D} Under the conditions of Theorem
\ref{Th Transmute} the transmutation operators $T_{f}$ and $T_{1/f}$ are
related by the expressions
\[
T_{1/f}[u](x)=\frac{1}{f(x)}\bigg(\int_{0}^{x}f(\eta)T_{f}[u^{\prime}%
](\eta)\,d\eta+u(0)\bigg)
\]
and
\[
T_{f}[u](x)=f(x)\bigg(\int_{0}^{x}\frac{1}{f(\eta)}T_{1/f}[u^{\prime}%
](\eta)\,d\eta+u(0)\bigg),
\]
valid for any $u\in C^{1}[-b,b]$, and their integral kernels $\mathbf{K}%
_{f}(x,t):=\mathbf{K}_{f}(x,t;h)$ and $\mathbf{K}_{1/f}(x,t):=\mathbf{K}%
_{1/f}(x,t;-h)$ are related by the expressions
\begin{equation}
\mathbf{K}_{1/f}(x,t)=-\frac{1}{f(x)}\bigg(\int_{-t}^{x}\partial_{t}%
\mathbf{K}_{f}(s,t)f(s)\,ds+\frac{f^{\prime}(0)}{2}f(-t)\bigg) \label{K2}%
\end{equation}
and
\begin{equation}
\mathbf{K}_{f}(x,t)=-f(x)\bigg(\int_{-t}^{x}\partial_{t}\mathbf{K}%
_{1/f}(s,t)\frac1{f(s)}\,ds-\frac{f^{\prime}(0)}{2f(-t)}\bigg) . \label{K1}%
\end{equation}
\end{proposition}

The following commutation relations immediately follow from Proposition
\ref{PropTT_D}.

\begin{corollary}
[\cite{KrT2012}]\label{Cor Commutation Relations}The following operator
equalities hold on $C^{1}[-b,b]$:
\begin{align*}
\partial_{x}fT_{1/f}  &  =fT_{f}\partial_{x}\\
\partial_{x}\frac{1}{f}T_{f}  &  =\frac{1}{f}T_{1/f}\partial_{x}.
\end{align*}
\end{corollary}

\begin{example}
[\cite{KrT2012}]\label{ModelExample} Take $f(x)=x+1$ and some segment
$[-b,b]\subset(-1,1)$. Then $q_{f}=0$ and $q_{1/f} = \frac2{(x+1)^{2}}$. For
the potential $q_{f}$ the transmutation operator $\mathbf{T}_{0}$ is obviously
the identity operator, hence $\mathbf{K}(x,t;0) = 0$ and from Proposition
\ref{PropChangeOfH} we obtain
\[
\mathbf{K}_{f}(x,t) := \mathbf{K}(x,t;1) = 1/2.
\]
Hence we get from \eqref{K2} that
\[
\mathbf{K}_{1/f}(x,t) = \frac{t-1}{2(x+1)}.
\]
It is easy to see that the integral kernels $\mathbf{K}_{f}$ and
$\mathbf{K}_{1/f}$ satisfy the Goursat problems \eqref{GoursatKh1},
\eqref{GoursatKh2} and relation \eqref{K1} holds.

For more examples of transmutation integral kernels we refer the reader to
\cite{KrT2012}.
\end{example}

\section{Goursat-to-Goursat transmutation operators}\label{SectGGTO}

Let $\widetilde{u}$ and $u$ be solutions of the equations $\square
\widetilde{u}=0$ and $\left(  \square-q(x)\right)  u=0$ in $\overline
{\mathbf{R}}$ where $\mathbf{R}$ is a square with the diagonal with endpoints $(b,b)$ and
$(-b,-b)$, respectively such that $u=T_{f}\widetilde{u}$. Consider the
operator $T_{G}$ mapping the Goursat data corresponding to $\widetilde{u}$
into the Goursat data corresponding to $u$,
\[
T_{G}:\quad\binom{\widetilde{u}(x,x)}{\widetilde{u}(x,-x)}\quad\longmapsto
\quad\binom{u(x,x)}{u(x,-x)}.
\]
Denote $\varphi(x):=\widetilde{u}(x,x)$, $\psi(x):=\widetilde{u}(x,-x)$,
$\Phi(x):=u(x,x)$, $\Psi(x):=u(x,-x)$. We have $\varphi(0)=\psi(0)$ and
$\Phi(0)=\Psi(0)$. The solution of the wave equation $\widetilde{u}$ is
represented via its Goursat data as follows $\widetilde{u}(x,t)=\varphi
(\frac{x+t}{2})+\psi(\frac{x-t}{2})-\varphi(0)$. Application of the operator
$T_{f}$ gives us the equalities%
\[
T_{f}\varphi\Bigl(\frac{x+t}{2}\Bigr)=\varphi\Bigl(\frac{x+t}{2}%
\Bigr)+\int_{-x}^{x}\mathbf{K}_{f}(x,\tau)\varphi\Bigl(\frac{\tau+t}%
{2}\Bigr)d\tau,
\]
\[
T_{f}\psi\Bigl(\frac{x-t}{2}\Bigr)=\psi\Bigl(\frac{x-t}{2}\Bigr)+\int_{-x}%
^{x}\mathbf{K}_{f} (x,\tau)\psi\Bigl(\frac{\tau-t}{2}\Bigr)d\tau
\]
and $T_{f}[\varphi(0)]=\varphi(0)f$. Considering the value of these
expressions when $t=x$ and $t=-x$ and introducing obvious changes of variables
we obtain the following relations%
\[
\Phi(x)=\varphi(x)+2\int_{0}^{x}\mathbf{K}_{f}(x,2t-x)\varphi(t)dt+\psi
(0)+2\int_{-x}^{0}\mathbf{K}_{f}(x,2t+x)\psi(t)dt-\varphi(0)f(x)
\]
and%
\[
\Psi(x)=\varphi(0)+2\int_{-x}^{0}\mathbf{K}_{f}(x,2t+x)\varphi(t)dt+\psi
(x)+2\int_{0}^{x}\mathbf{K}_{f}(x,2t-x)\psi(t)dt-\varphi(0)f(x).
\]
Due to the boundedness of the operators $T_{f}$ and $T_{f}^{-1}$ as well as to
the continuous dependence of the solutions of the Goursat problems under
consideration on their respective Goursat data (see, e.g., \cite[Sect.
15]{Vladimirov}) the operator $T_{G}$ together with $T_{G}^{-1}$ are bounded
on vector functions from $C^{1}[-b,b]\times C^{1}[-b,b]$ equipped with a
suitable, for example, maximum, norm.

It is easy to establish the following mapping properties of $T_{G}$,%
\begin{equation}
T_{G}:\quad2^{n-1}\binom{x^{n}}{x^{n}}\quad\longmapsto\quad\binom
{u_{2n-1}(x,x)}{u_{2n-1}(x,x)} \label{mapping1}%
\end{equation}
and
\begin{equation}
T_{G}:\quad2^{n-1}\binom{x^{n}}{-x^{n}}\quad\longmapsto\quad\binom
{u_{2n}(x,x)}{-u_{2n}(x,x)}. \label{mapping2}%
\end{equation}
Indeed, consider the wave polynomial $p_{2n-1}(x,t)$, $n=1,2,\ldots$, see
\eqref{WavePolynomials}. Its Goursat data have the form $p_{2n-1}%
(x,x)=p_{2n-1}(x,-x)=2^{n-1}x^{n}$. The image of $p_{2n-1}$ under the action
of $T_{f}$ is $u_{2n-1}$, see \eqref{Mapping Wave Polynomial}. Using the
property \eqref{umParity} we obtain (\ref{mapping1}). Analogously,
consideration of $p_{2n}$ leads to (\ref{mapping2}).

\begin{proposition}
\label{Prop Unif Conv u} Let $u$ be a regular solution of the equation
$\left(  \square-q(x)\right)  u=0$ in $\overline{\mathbf{R}}$ such that its
Goursat data admit the following series expansions%
\begin{equation}
\frac{1}{2}\left(  u(x,x)+u(x,-x)\right)  =c_{0}u_{0}(x,x)+\sum_{n=1}^{\infty
}c_{n}u_{2n-1}(x,x), \label{+1}%
\end{equation}
and
\begin{equation}
\frac{1}{2}\left(  u(x,x)-u(x,-x)\right)  =\sum_{n=1}^{\infty}b_{n}%
u_{2n}(x,x), \label{-1}%
\end{equation}
both uniformly convergent on $[-b,b]$. Then for any $(x,t)\in\overline
{\mathbf{R}}$,
\begin{equation}
u(x,t)=c_{0}u_{0}(x,t)+\sum_{n=1}^{\infty}\left(  c_{n}u_{2n-1}(x,t)+b_{n}%
u_{2n}(x,t)\right)  \label{u}%
\end{equation}
and the series converges uniformly in $\overline{\mathbf{R}}$.
\end{proposition}

\begin{proof}
Notice that if (\ref{u}) is valid then (\ref{+1}) and (\ref{-1}) hold. Indeed,
due to \eqref{umParity}
\[%
\begin{split}
\frac{1}{2}\left(  u(x,x)+u(x,-x)\right)   &  =c_{0}u_{0}(x,x)\\
&  +\frac{1}{2} \sum_{n=1}^{\infty}\left(  c_{n}\left(  u_{2n-1}%
(x,x)+u_{2n-1}(x,x)\right)  +b_{n}\left(  u_{2n} (x,x)-u_{2n}(x,x)\right)
\right) \\
&  =c_{0}u_{0}(x,x)+\sum_{n=1}^{\infty}c_{n}u_{2n-1}(x,x)
\end{split}
\]
and similarly (\ref{-1}) can be verified.

Suppose that (\ref{+1}) and (\ref{-1}) are valid. Then we have
\[
\binom{u(x,x)}{u(x,-x)}=c_{0}\binom{u_{0}(x,x)}{u_{0}(x,x)}+\sum_{n=1}%
^{\infty}\left(  c_{n}\binom{u_{2n-1}(x,x)}{u_{2n-1}(x,x)}+b_{n}\binom
{u_{2n}(x,x)}{-u_{2n}(x,x)}\right)  .
\]
From (\ref{mapping1}) and (\ref{mapping2}) we obtain
\begin{equation}
\label{phipsi}%
\begin{split}
\binom{\varphi(x)}{\psi(x)}  &  := T_{G}^{-1}\binom{u(x,x)}{u(x,-x)}%
=c_{0}\binom{1}{1}+\sum_{n=1}^{\infty}2^{n-1}\left(  c_{n}\binom{x^{n}}{x^{n}%
}+b_{n}\binom{x^{n}}{-x^{n}}\right) \\
&  =\binom{c_{0}+\sum_{n=1}^{\infty}2^{n-1}\left(  c_{n}+b_{n}\right)  x^{n}%
}{c_{0}+\sum_{n=1}^{\infty}2^{n-1}\left(  c_{n}-b_{n}\right)  x^{n}}.
\end{split}
\end{equation}
Under the conditions of the proposition and due to the boundedness of
$T_{G}^{-1}$ we have that both series in the last equality are uniformly
convergent on $[-b,b]$. Then applying Proposition 1 from \cite{KKTT} we obtain
that the unique solution $\widetilde{u}$ of the Goursat problem for the wave
equation in $\overline{\mathbf{R}}$ and with the Goursat data $\binom
{\varphi(x)}{\psi(x)}$ has the form
\begin{equation}
\widetilde{u}(x,t)=c_{0}p_{0}(x,t)+\sum_{n=1}^{\infty}\left(  c_{n}%
p_{2n-1}(x,t)+b_{n}p_{2n}(x,t)\right)  \label{util}%
\end{equation}
and the series converges uniformly in $\overline{\mathbf{R}}$. Since
$u=T_{f}\widetilde{u}$, $T_{f}p_{n}=u_{n}$ and $T_{f}$ is bounded we obtain
the uniform convergence of the series (\ref{u}).
\end{proof}

\section{Transmutations for the Vekua operator}\label{SectTransmutVekua}

By analogy with \cite{CKM AACA} we introduce the following pair of operators
\begin{equation}
\mathbf{V}_{1}=T_{f}\mathcal{R}+\mathbf{j}T_{1/f}\mathcal{I} \label{Tbold1}%
\end{equation}
and
\begin{equation}
\mathbf{V}_{2}=T_{1/f}\mathcal{R}+\mathbf{j}T_{f}\mathcal{I}. \label{Tbold2}%
\end{equation}

\begin{proposition}
\label{Prop Transm dz} The following equalities hold for any $\mathbb{B}%
$-valued, continuously differentiable function $w$ defined on $\overline
{\mathbf{R}}$.
\[
\left(  \partial_{\overline{z}}-\frac{\partial_{\overline{z}}f}{f}C\right)
\mathbf{V}_{1}w=\mathbf{V}_{2}\left(  \partial_{\overline{z}}w\right)
,\qquad\left(  \partial_{\overline{z}}+\frac{\partial_{z}f}{f}C\right)
\mathbf{V}_{2}w=\mathbf{V}_{1}\left(  \partial_{\overline{z}}w\right)  .
\]%
\[
\left(  \partial_{z}-\frac{\partial_{z}f}{f}C\right)  \mathbf{V}%
_{1}w=\mathbf{V}_{2}\left(  \partial_{z}w\right)  ,\qquad\left(  \partial
_{z}+\frac{\partial_{\overline{z}}f}{f}C\right)  \mathbf{V}_{2}w=\mathbf{V}%
_{1}\left(  \partial_{z}w\right)  .
\]

\end{proposition}

\begin{proof}
The proof consists in a direct calculation with the aid of the commutation
relations from Corollary \ref{Cor Commutation Relations}. For example, for
$w=u+\mathbf{j}v$ we have
\begin{align*}
\left(  \partial_{\overline{z}}-\frac{\partial_{\overline{z}}f}{f}C\right)
\mathbf{V}_{1}w  &  =\frac{1}{2}\left(  f\partial_{\overline{z}}\left(
\frac{1}{f}\mathcal{R}\left(  \mathbf{V}_{1}w\right)  \right)  +\mathbf{j}%
\frac{1}{f}\partial_{\overline{z}}\left(  f\mathcal{I}\left(  \mathbf{V}%
_{1}w\right)  \right)  \right) \\
&  =\frac{1}{2}\left(  f\partial_{x}\left(  \frac{1}{f}T_{f}u\right)
-\mathbf{j}\partial_{t}T_{f}u+\mathbf{j}\frac{1}{f}\partial_{x}\left(
fT_{1/f}v\right)  -\partial_{t}T_{1/f}v\right) \\
&  =\frac{1}{2}\left(  T_{1/f}\partial_{x}u-\mathbf{j}T_{f}\partial
_{t}u+\mathbf{j}T_{f}\partial_{x}v-T_{1/f}\partial_{t}v\right)  =\mathbf{V}%
_{2}\left(  \partial_{\overline{z}}w\right)  .
\end{align*}
\end{proof}

An immediate corollary from Proposition \ref{Prop Transm dz} is the fact that
the operators $\mathbf{V}_{1}$ and $\mathbf{V}_{2}$ map the hyperbolic
analytic functions satisfying (\ref{hCR}) into the solutions of
(\ref{Vekua main hyper}) and (\ref{Vekua successor}) respectively. Moreover,
they map powers of the variable $z$ into the corresponding formal powers.

\begin{proposition}
\label{PropMapComplexPowers}For any $z\in\overline{\mathbf{R}}$,
$n\in\mathbb{N}_{0}$ and $a\in\mathbb{B}$ the following equalities are valid
\[
\mathbf{V}_{1}[az^{n}]=Z^{(n)}(a,0;z)\quad\text{and}\quad\mathbf{V}_{2}%
[az^{n}]=Z_{1}^{(n)}(a,0;z).
\]
\end{proposition}

\begin{proof}
The proof consists in the observation that for $a=a^{\prime}+\mathbf{j}%
b^{\prime}$ and $z=x+\mathbf{j}t$ one has
\[
az^{n}=\left(  a^{\prime}+\mathbf{j}b^{\prime}\right)  \sum_{m=0}^{n}
\binom{n}{m}x^{n-m}\mathbf{j}^{m}t^{m}%
\]
and the result follows from the formulas \eqref{Zn}, \eqref{Znodd},
\eqref{Zneven} by application of the mapping properties
\eqref{mapping powers 1}, \eqref{mapping powers 2}.
\end{proof}

Notice that both operators $\mathbf{V}_{1}$ and $\mathbf{V}_{2}$ are bounded
on the space of continuous functions. Indeed, consider
\begin{align*}
\left\vert \mathbf{V}_{1}w\right\vert  &  =\frac{1}{2}\left(  \left\vert
\left(  \mathbf{V}_{1}w\right)  ^{+}\right\vert _{\mathbb{C}}+\left\vert
\left(  \mathbf{V}_{1}w\right)  ^{-}\right\vert _{\mathbb{C}}\right) \\
&  =\frac{1}{2}\left(  \left\vert \mathcal{R}\left(  \mathbf{V}_{1}w\right)
+\mathcal{I}\left(  \mathbf{V}_{1}w\right)  \right\vert _{\mathbb{C}%
}+\left\vert \mathcal{R}\left(  \mathbf{V}_{1}w\right)  -\mathcal{I}\left(
\mathbf{V}_{1}w\right)  \right\vert _{\mathbb{C}}\right) \\
&  \leq\left\vert \mathcal{R}\left(  \mathbf{V}_{1}w\right)  \right\vert
_{\mathbb{C}}+\left\vert \mathcal{I}\left(  \mathbf{V}_{1}w\right)
\right\vert _{\mathbb{C}}=\left\vert T_{f}u\right\vert _{\mathbb{C}%
}+\left\vert T_{1/f}v\right\vert _{\mathbb{C}}.
\end{align*}
From the boundedness of the operators $T_{f}$ and $T_{1/f}$ we have
\[
\max\left\vert \mathbf{V}_{1}w\right\vert \leq M\left(  \max\left\vert
u\right\vert _{\mathbb{C}}+\max\left\vert v\right\vert _{\mathbb{C}}\right)
\]
where $M=\max\left\{  \left\Vert T_{f}\right\Vert ,\left\Vert T_{1/f}%
\right\Vert \right\}  $. \ Since $\left\vert u\right\vert _{\mathbb{C}}%
\leq\frac{1}{2}\left(  \left\vert w^{+}\right\vert _{\mathbb{C}}+\left\vert
w^{-}\right\vert _{\mathbb{C}}\right)  $ and $\left\vert v\right\vert
_{\mathbb{C}}\leq\frac{1}{2}\left(  \left\vert w^{+}\right\vert _{\mathbb{C}%
}+\left\vert w^{-}\right\vert _{\mathbb{C}}\right)  $, we obtain
$\max\left\vert \mathbf{V}_{1}w\right\vert \leq2M\max\left\vert w\right\vert
$. The boundedness of $\mathbf{V}_{2}$ is proved analogously.

The inverse operators $\mathbf{V}_{1}^{-1}$ and $\mathbf{V}_{2}^{-1}$ have the
form
\[
\mathbf{V}^{-1}_{1}=T^{-1}_{f}\mathcal{R}+\mathbf{j}T^{-1}_{1/f}%
\mathcal{I}\qquad\text{and}\qquad\mathbf{V}^{-1}_{2}=T^{-1}_{1/f}%
\mathcal{R}+\mathbf{j}T^{-1}_{f}\mathcal{I}%
\]
and clearly are bounded. For the explicit formulae for $T^{-1}_{f}$ and
$T^{-1}_{1/f}$ we refer to \cite[Theorem 10]{KrT2012}.

The following statement establishes a relation between the derivatives of
hyperbolic analytic functions and the generalized derivatives of their images
under the action of the transmutation operator.

\begin{proposition}
\label{PropDerivatives}Let $w$ be a hyperbolic analytic function in
$\mathbf{R}$ and $W=\mathbf{V}_{1}w$ be a corresponding solution of
\eqref{Vekua main hyper}. Then whenever the corresponding derivative of $w$
exists there exists a derivative of the same order in the sense of Bers of the
function $W$ and vice versa, and the following relations are valid
\begin{equation}
\mathbf{V}_{1}\left(  \partial_{z}^{(2n)}w\right)  =W^{\left[  2n\right]
}\quad\text{and\quad}\mathbf{V}_{2}\left(  \partial_{z}^{(2n-1)}w\right)
=W^{\left[  2n-1\right]  }\text{,\quad}n=1,2,\ldots
.\label{Relations for derivatives}%
\end{equation}
\end{proposition}

\begin{proof}
From Proposition \ref{Prop Transm dz} we have%
\begin{equation}
\overset{\circ}{W}=\mathbf{V}_{2}\left(  \partial_{z}w\right)  . \label{45}%
\end{equation}
$\overset{\circ}{W}$ is a solution of the succeeding Vekua equation
(\ref{Vekua successor}). Denote $W_{1}=\overset{\circ}{W}$. Any solution of
(\ref{Vekua successor}) is the image of a bicomplex analytic function under
the action of the operator $\mathbf{V}_{2}$, so $W_{1}=\mathbf{V}_{2}w_{1}$.
Due to Proposition \ref{Prop Transm dz} we have $\overset{\circ}{W}%
_{1}=\mathbf{V}_{1}\left(  \partial_{z}w_{1}\right)  $. Thus, $\overset
{\circ\circ}{W}=\mathbf{V}_{1}\left(  \partial_{z}^{2}w\right)  $ because from
(\ref{45}) $w_{1}=\partial_{z}w$. Now (\ref{Relations for derivatives}) can be
easily proved by induction.
\end{proof}

\section{Expansion theorem}\label{SectETh}

\begin{theorem}
\label{Th W series} A solution $W$ of \eqref{Vekua main x} can be represented
as a uniformly convergent series
\begin{equation}
W(z)=u(z)+\mathbf{j}v(z)=\sum_{n=0}^{\infty} Z^{(n)}(a_{n},0;z)
\label{W=Taylor}%
\end{equation}
in $\overline{\mathbf{R}}$ iff the functions $u(x,x)$ and $u(x,-x)$ admit the
uniformly convergent series expansions on $[-b,b]$ of the form \eqref{+1} and
\eqref{-1} where the coefficients of the expansions are related by the
equalities
\[
a_{n}=c_{n}+\mathbf{j}b_{n},\quad n=0,1,\ldots
\]
and the coefficient $b_{0}$ defines the value of $v$ at the origin.
\end{theorem}

\begin{proof}
Suppose that $W$ is a solution of (\ref{Vekua main x}) such that
(\ref{W=Taylor}) holds. Denote $c_{n}:=\mathcal{R}(a_{n})$ and $b_{n}%
:=\mathcal{I}(a_{n})$, $n=0,1,\ldots$. By \eqref{ZasWavePol} we have
\[
Z^{(n)}(1,0;z)=u_{2n-1}(x,t)+\mathbf{j}v_{2n}(x,t),\quad n=1,2,\ldots
\]
and
\[
Z^{(n)}(\mathbf{j},0;z) =u_{2n}(x,t)+\mathbf{j}v_{2n-1}(x,t),\quad
n=1,2,\ldots.
\]
Then
\[
W(z)=c_{0}u_{0}(x,t)+b_{0}\mathbf{j}v_{0}(x,t)+\sum_{n=1}^{\infty}\left(
c_{n}\left(  u_{2n-1}(x,t)+\mathbf{j}v_{2n}(x,t)\right)  +b_{n}\left(
u_{2n}(x,t)+\mathbf{j}v_{2n-1}(x,t)\right)  \right)  .
\]
Hence
\[
u:=\mathcal{R}(W)=c_{0}u_{0}+\sum_{n=1}^{\infty}\left(  c_{n}u_{2n-1}%
+b_{n}u_{2n}\right)
\]
and
\[
v:=\mathcal{I}(W)=b_{0}v_{0}+\sum_{n=1}^{\infty}\left(  c_{n}v_{2n}%
+b_{n}v_{2n-1}\right)  .
\]
Thus,
\[
u(x,x)=c_{0}u_{0}(x,x)+\sum_{n=1}^{\infty}\left(  c_{n}u_{2n-1}(x,x)+b_{n}%
u_{2n}(x,x)\right)
\]
and
\[
u(x,-x)=c_{0}u_{0}(x,x)+\sum_{n=1}^{\infty}\left(  c_{n}u_{2n-1}%
(x,x)-b_{n}u_{2n}(x,x)\right)  .
\]
We obtain that (\ref{+1}) and (\ref{-1}) hold.

Now let us assume that $W$ is a solution of (\ref{Vekua main x}) in
$\overline{\mathbf{R}}$ such that for $u=\mathcal{R}(W)$ the equalities
(\ref{+1}) and (\ref{-1}) hold. Then due to Proposition \ref{Prop Unif Conv u}
the function $u$ admits a series expansion (\ref{u}), uniformly convergent in
$\overline{\mathbf{R}}$ and at the same time $\widetilde{u}=T_{f}^{-1}u$ is a
solution of a Goursat problem for the wave equation with the Goursat data
$\binom{\varphi}{\psi}$ defined by (\ref{phipsi}). $\widetilde{u}$ has the
form (\ref{util}).

Now, let us consider the Goursat problem (\ref{hCR}), (\ref{condG}) with
\[
\Phi(x):=2\varphi(x)-c_{0}+b_{0}\quad\text{and}\quad\Psi(x):=2\psi
(x)-c_{0}-b_{0}%
\]
where $b_{0}\in\mathbb{C}$ is an arbitrary number. Thus,
\[
\Phi(x)=\sum_{n=0}^{\infty}2^{n}\left(  c_{n}+b_{n}\right)  x^{n}%
\quad\text{and}\quad\Psi(x)=\sum_{n=0}^{\infty}2^{n}\left(  c_{n}%
-b_{n}\right)  x^{n}.
\]
Due to Proposition \ref{Prop Goursat hCR}, the unique solution of this Goursat
problem has the form $\widetilde{W}(z)=\sum\limits_{n=0}^{\infty}a_{n}z^{n}$
where
\[
a_{n}=P^{+}\left(  c_{n}+b_{n}\right)  +P^{-}\left(  c_{n}-b_{n}\right)
=c_{n}+\mathbf{j}b_{n},\quad n=0,1,\ldots.
\]
Application of the operator $\mathbf{V}_{1}$ to $\widetilde{W}$ gives us a
solution $W$ of (\ref{Vekua main x}) in $\overline{\mathbf{R}}$ in the form of
a uniformly convergent series (\ref{W=Taylor}) with $u=\mathcal{R}(W)$.
\end{proof}

\begin{theorem}
\label{Th Taylor Vekua} Let $W$ be a solution of \eqref{Vekua main x}
admitting in $\overline{\mathbf{R}}$ a uniformly convergent series expansion
\eqref{W=Taylor}. Then there exist its derivatives in the sense of Bers of any
order and the Taylor formula for the expansion coefficients holds,
\begin{equation}
a_{n}=\frac{W^{[n]}(0)}{n!}. \label{TaylorCoef}%
\end{equation}
\end{theorem}

\begin{proof}
Under the condition of the theorem we obtain that $w=\mathbf{V}_{1}^{-1}W$ has
the form (\ref{power series}) with the expansion coefficients
(\ref{Taylor coefs}). The derivatives of $w$ of any order exist because they
reduce to the differentiation of power series expansions of $\Phi$ and $\Psi$
from Proposition \ref{Prop Goursat hCR}. Application of the operator
$\mathbf{V}_{1}$ to the even derivatives and of $\mathbf{V}_{2}$ to the odd,
due to Proposition \ref{PropDerivatives} gives us the corresponding
derivatives of $W$ in the sense of Bers as well as the formula
(\ref{TaylorCoef}).
\end{proof}

Even if the conditions of Theorem \ref{Th W series} are not fulfilled and a
solution of (\ref{Vekua main x}) cannot be represented globally as a uniformly
convergent formal power series of the form (\ref{W=Taylor}), it may be
arbitrarily closely approximated by finite linear combinations of the formal
powers. The following analogue of the Runge approximation theorem from
classical complex analysis is a corollary of the existence of the
transmutation operator $\mathbf{V}_{1}$.

\begin{proposition}
\label{Prop RungeZ} Let $W$ be a solution of \eqref{Vekua main x} in
$\mathbf{R}$. Then there exists a sequence of polynomials in formal powers
$P_{N}=\sum_{n=0}^{N}Z^{(n)}(a_{n}, 0; z)$ uniformly convergent to $W$ in
$\overline{\mathbf{R}}$.
\end{proposition}

\begin{proof}
Consider $w = \mathbf{V}_{1}^{-1}W$. By Proposition \ref{Prop Transm dz} we
have $\mathbf{V}_{2}(\partial_{\bar z} w) = 0$, hence $w$ is a solution of
\eqref{hCR} in $\overline{\mathbf{R}}$. Despite the domain $\overline
{\mathbf{R}}$ is smaller than the domain $\overline{R}$, the proof of
Proposition \ref{Prop RungeHCR} may be applied without changes to the function
$w$. Hence for an arbitrary $\varepsilon>0$ there exists a polynomial
$p_{N}=\sum_{n=0}^{N} a_{n}z^{n}$ such that for any
$(x,t)\in\overline{\mathbf{R}}$
\[
|w(x,t) - p_{N}(x,t)|<\frac{\varepsilon}{\|\mathbf{V}_{1}\|},
\]
where $\|\mathbf{V}_{1}\|<\infty$ is the norm of the transmutation operator
$\mathbf{V}_{1}$. Applying $\mathbf{V}_{1}$ and Proposition
\ref{PropMapComplexPowers} we obtain that
\[
\left|  W(x,t) - \sum_{n=0}^{N} Z^{(n)}(a_{n},0;z)\right|
\le\|\mathbf{V}_{1}\|\cdot\max_{(x,t)\in\overline{\mathbf{R}}}|w(x,t) -
p_{N}(x,t)| <\varepsilon.
\]
\end{proof}

\section{Integral kernels as solutions of Vekua equations}\label{SectKernelAsSolVekua}
In Section \ref{SectTransmutVekua} we showed that
combinations \eqref{Tbold1} and \eqref{Tbold2} of the transmutation operators
$T_{f}$ and $T_{1/f}$ are related with Vekua operators and formal powers. In
this section we consider the combination
\begin{equation}
\label{KVekua}\mathbf{K} = \mathbf{K}_{f} - \mathbf{j} \mathbf{K}_{1/f}%
\end{equation}
of the integral kernels of the operators $T_{f}$ and $T_{1/f}$.

\begin{theorem}
\label{Th K Vekua} The function $\mathbf{K}$ given by \eqref{KVekua} is a
solution of the hyperbolic Vekua equation \eqref{Vekua main x}.
\end{theorem}

\begin{proof}
The proof immediately follows from relations \eqref{K2} and \eqref{K1} and
conditions \eqref{Vek1}, \eqref{Vek2}.
\end{proof}

The way of constructing the integral kernel $\mathbf{K}_{1/f}$ by the integral
kernel $\mathbf{K}_{f}$ given by \eqref{K2} requires the knowledge of the
kernel $\mathbf{K}_{f}(x,t)$ not only on a natural domain $|t|\leq|x|\leq b$
sufficient for defining the transmutation operator, but on a larger domain
$|t|\leq b$, $|x|\leq b$. Pseudoanalytic function theory allows us to obtain
another method of reconstructing $\mathbf{K}_{1/f}$ by the known
$\mathbf{K}_{f}$ and vice versa. Resulted formulae are simpler than those
mentioned in \cite[Remark 21]{KrT2012}.

\begin{corollary}
\label{Cor K Vekua} The integral kernels $\mathbf{K}_{f}$ and $\mathbf{K}%
_{1/f}$ of the transmutation operators $T_{f}$ and $T_{1/f}$ are related by
the expressions
\begin{multline}
\label{K2Vekua}\mathbf{K}_{1/f}(x,t) = -\frac{1}{f(x)} \int_{0}^{x} f(\eta)
\partial_{t} \mathbf{K}_{f}(\eta,0)\,d\eta \\
- \int_{0}^{t} \partial_{x}
\mathbf{K}_{f}(x,\xi)\,d\xi+ \frac{f^{\prime}(x)}{f(x)}\int_{0}^{t}
\mathbf{K}_{f}(x,\xi)\,d\xi- \frac{f^{\prime}(0)}{2f(x)}%
\end{multline}
and
\begin{multline}
\label{K1Vekua}\mathbf{K}_{f}(x,t) = -f(x) \int_{0}^{x} \frac{1}{f(\eta)}
\partial_{t} \mathbf{K}_{1/f}(\eta,0)\,d\eta\\
- \int_{0}^{t} \partial_{x}
\mathbf{K}_{1/f}(x,\xi)\,d\xi- \frac{f^{\prime}(x)}{f(x)}\int_{0}^{t}
\mathbf{K}_{1/f}(x,\xi)\,d\xi+ \frac{f^{\prime}(0)}{2}f(x).
\end{multline}
\end{corollary}

\begin{proof}
Suppose first that $q_{f}:= \frac{f^{\prime\prime}}{f} \in C^{1}[-b,b]$. In
such case we also have $q_{1/f}= 2\Bigl(\frac{f^{\prime}}{f}\Bigr)^{2} -
q_{f}\in C^{1}[-b,b]$ and the integral kernels $\mathbf{K}_{f}$ and
$\mathbf{K}_{1/f}$ are twice continuously differentiable with respect to both
variables and satisfy equation \eqref{GoursatKh1} with potentials $q_{f}$ and
$q_{1/f}$ correspondingly. Hence we obtain from Theorems \ref{ThDarboux} and
\ref{Th K Vekua} for the function $\mathbf{K}_{f}$ using the operator
$\overline{A}_{h}$ given by \eqref{ABarKernel} that the function
$\mathbf{K}_{1/f}$ differs from the function
\[
-v = -\frac{1}{f(x)} \int_{0}^{x} f(\eta) \partial_{t} \mathbf{K}_{f}%
(\eta,0)\,d\eta- \int_{0}^{t} \partial_{x} \mathbf{K}_{f}(x,\xi)\,d\xi+
\frac{f^{\prime}(x)}{f(x)}\int_{0}^{t} \mathbf{K}_{f}(x,\xi)\,d\xi
\]
by the term $cf^{-1}$, where $c$ is a constant. From the condition
$\mathbf{K}_{1/f}(0,0) = -\frac{f^{\prime}(0)}2$ (condition
\eqref{GoursatKh2}) we find that $c=-\frac{f^{\prime}(0)}2$.

Suppose now that $q_{f}\in C[-b,b]$. We proceed as follows. Consider a
sequence of polynomials $\{f_{n}\}_{n\in\mathbb{N}}$ such that $f_{n}(0)=1$
and $f_{n}^{\prime\prime}\to f^{\prime\prime}$, $f_{n}^{\prime}\to f^{\prime}$
and $f_{n}\to f$ uniformly on $[-b,b]$ as $n\to\infty$. See, e.g., \cite[Proof
of Theorem 11]{CKT} for a possible construction. Since the function $f$ does
not vanish on $[-b,b]$ we may additionally assume that all functions $f_{n}$
do no vanish on $[-b,b]$. It is easy to see that $q_{n}:=\frac{f_{n}%
^{\prime\prime}}{f_{n}}\to q$ uniformly and $q_{n}\in C^{1}[-b,b]$. Every
$f_{n}$ is associated with the transmutation operators $T_{f_{n}}$ and
$T_{1/f_{n}}$ with the integral kernels $\mathbf{K}_{f_{n}}$ and
$\mathbf{K}_{1/f_{n}}$. Consider functions $\mathbf{H}_{f_{n}}(u,v) =
\mathbf{K}_{f_{n}}(u+v, u-v; f^{\prime}_{n}(0))$. It is well known (e.g.,
\cite{Vladimirov}) that the Goursat problem \eqref{GoursatTh1},
\eqref{GoursatTh2} is equivalent to either of the systems of integral
equations
\[
\begin{cases}
\mathbf{H}(u,v) = \frac h2 + \int_{0}^{u} \mathbf{G}(u^{\prime}, v)
\,du^{\prime},\\
\mathbf{G}(u,v) = \frac12 q(u)+\int_{0}^{v} q(u+v^{\prime}) \mathbf{H}(u,
v^{\prime})\,dv^{\prime}%
\end{cases}
\]
and
\[
\begin{cases}
\mathbf{H}(u,v) = \frac h2 + \frac12\int_{0}^{u} q(s)\,ds+\int_{0}^{v}
\mathbf{F}(u, v^{\prime}) \,dv^{\prime},\\
\mathbf{F}(u,v) = \int_{0}^{u} q(u^{\prime}+v) \mathbf{H}(u^{\prime},
v)\,du^{\prime},
\end{cases}
\]
where $\mathbf{G}=\partial_{u} \mathbf{H}$ and $\mathbf{F}=\partial_{v}
\mathbf{H}$, from which we conclude that the functions $\mathbf{H}_{f_{n}}$,
$\partial_{u} \mathbf{H}_{f_{n}}$ and $\partial_{v} \mathbf{H}_{f_{n}}$ are
uniformly convergent to the functions $\mathbf{H}_{f}$, $\partial
_{u}\mathbf{H}_{f}$ and $\partial_{v}\mathbf{H}_{f}$, respectively. Hence the
partial derivatives of $\mathbf{K}_{f_{n}}$ are uniformly convergent as well,
and taking in \eqref{K2Vekua} the limit as $n\to\infty$ we finish the proof.

The second formula \eqref{K1Vekua} easily follows from \eqref{K2Vekua} by
changing $f$ to $1/f$.
\end{proof}

\begin{example}
\label{Example K Vekua} For the functions $\mathbf{K}_{f}$ and $\mathbf{K}%
_{1/f}$ from Example \ref{ModelExample} we get $\mathbf{K} = \frac12 -
\mathbf{j} \frac{t-1}{2(x+1)}$ and
\[
\partial_{\bar z} \mathbf{K} = \frac12 \left(  \mathbf{j} \frac{t-1}%
{2(x+1)^{2}} + \frac1{2(x+1)}\right)  = \frac1{2(x+1)}\mathbf{\overline{K}} =
\frac{f^{\prime}(x)}{2f(x)}\mathbf{\overline{K}}.
\]
Formulae \eqref{K2Vekua} and \eqref{K1Vekua} can be verified as well.
\end{example}

\section{Formulae for integral kernels of transmutation operators}\label{SectFIK}

It follows from Theorems \ref{Th Taylor Vekua} and \ref{Th K Vekua} and from
\eqref{nth Der} that the knowledge of the derivatives $\partial_{t}%
^{n}\mathbf{K}(0,0)$ allows one to obtain the generalized Taylor coefficients
\eqref{TaylorCoef} for $W = \mathbf{K}$. In this section we present several
formulae for computing these derivatives whenever the derivatives $q^{(m)}%
(0)$, $m\le n-1$ are known.

First we illustrate the expansion theorem with an example.

\begin{example}
For the function $\mathbf{K}$ from Example \ref{Example K Vekua} by
\eqref{nth Der} we obtain
\[
\mathbf{K}^{[0]}(0,0) = \frac12+\frac12\mathbf{j},\quad\mathbf{K}^{[1]}(0,0) =
\left.  -\frac{1}{2(x+1)}\right|  _{(0,0)}=-\frac12+0\mathbf{j},\quad
\mathbf{K}^{[n+1]} = \mathbf{j}\partial_{t}\mathbf{K}^{[n]} = 0,\ n\ge1.
\]
Hence the expansion coefficients are $a_{0}= \frac12+\frac12\mathbf{j}$,
$a_{1} = -\frac12+0\mathbf{j}$ and $a_{n}=0$ for $n\ge2$.

The first two functions of the systems $\{\varphi_{k}\}$ and $\{\psi_{k}\}$
are equal to
\begin{align*}
\varphi_{0}  &  = x+1, & \psi_{0}  &  = \frac1{x+1},\\
\varphi_{1}  &  = x, & \psi_{1}  &  = \frac{x^{3}+3x^{2}+3x}{3(x+1)},
\end{align*}
the first two formal powers are given by
\begin{align*}
Z^{(0)}(a,0; z)  &  = \mathcal{R}(a)\cdot\varphi_{0}+\mathbf{j}\mathcal{I}%
(a)\cdot\psi_{0},\\
Z^{(1)}(a,0; z)  &  = \mathcal{R}(a)\cdot(\varphi_{1}+\mathbf{j}t\psi
_{0})+\mathbf{j}\mathcal{I} (a)\cdot(\psi_{1}+\mathbf{j}t\varphi_{0})
\end{align*}
and we find that indeed,
\[
\mathbf{K}(x,t) = Z^{(0)}\Bigl(\frac12+\frac12\mathbf{j}, 0; z\Bigr)+Z^{(1)}%
\Bigl(-\frac12+0\mathbf{j}, 0; z\Bigr).
\]
\end{example}

We are able to calculate the expansion coefficients $a_{n}$ whenever we know
values of the derivatives $\partial_{t}^{n}\mathbf{K}(0,0)$. Recall that the
function $\mathbf{K}_{f}$ is the integral kernel of a transmutation operator
if and only if the function $\mathbf{H}_{f}(u,v) = \mathbf{K}_{f}(u+v,u-v)$ is
the solution of the Goursat problem \eqref{GoursatTh1}, \eqref{GoursatTh2}. We
have
\begin{equation}
\label{dtKdudvH}\partial_{t}^{n}\mathbf{K}_{f}(x,t) = \partial_{t}%
^{n}\mathbf{H}_{f}\Bigl(\frac{x+t}2, \frac{x-t}2\Bigr) = \frac1{2^{n}%
}(\partial_{u}-\partial_{v})^{n}\mathbf{H}_{f}(u,v)\Bigr|_{u=\frac
{x+t}2,\,v=\frac{x-t}2}.
\end{equation}
Moreover, it follows from \eqref{GoursatTh2} that
\begin{equation}
\label{duH(0)}\partial_{u}^{n}\mathbf{H}_{f}(u,0) =
\begin{cases}
\frac{f^{\prime}(0)}2 + \frac12\int_{0}^{u} q_{f}(s)\,ds, & n=0,\\
\frac12 q_{f}^{(n-1)}(u), & n\ge1,
\end{cases}
\end{equation}
and
\begin{equation}
\label{dvH(0)}\partial_{v}^{n}\mathbf{H}_{f}(0,v) =
\begin{cases}
\frac{f^{\prime}(0)}2, & n=0,\\
0, & n\ge1.
\end{cases}
\end{equation}
Hence to calculate the value of $\partial_{t}^{n}\mathbf{K}_{f}(0,0)$ it is
sufficient to transform the terms $\partial_{u}^{k}\partial_{v}^{n-k}%
\mathbf{H}_{f}$, where $1\le k\le n-1$, into the terms involving only
derivatives of the form $\partial_{u}^{\ell}\mathbf{H}_{f}$ or $\partial
_{v}^{m} \mathbf{H}_{f}$ and some derivatives of the potential $q_{f}$. Such
transformation may be performed using equation \eqref{GoursatTh1} which allows
us to reduce the orders $k$ and $n-k$ in the term $\partial_{u}^{k}%
\partial_{v}^{n-k}\mathbf{H}_{f}$ by 1. After several applications of equation
\eqref{GoursatTh1} and the Leibniz rule we obtain a finite sum of terms of the
form
\begin{equation}
\label{dtHGeneralTerm}q_{f}^{(n_{1})}(u+v)\cdot\ldots\cdot q_{f}^{(n_{\ell}%
)}(u+v) \partial^{d}_{u/v}\mathbf{H}_{f}(u,v).
\end{equation}
It is easy to see that the number $\ell$ of factors $q^{(n_{i})}(u+v)$
satisfies the inequality $0\le\ell\le\min(k,n-k)$ and that $n_{1}%
+\ldots+n_{\ell}+2\ell+d=n$.

Moreover, consider a term of the form \eqref{dtHGeneralTerm} obtained as a
result of application to $\mathbf{H}_{f}$ of some derivatives $\partial_{u}$
and $\partial_{v}$, equation \eqref{GoursatTh1} and the Leibniz rule. If
during this process we change every derivative $\partial_{u}$ by
$-\partial_{v}$ and every derivative $\partial_{v}$ by $-\partial_{u}$, we
obtain the term of exactly the same form, with the only difference that the
last derivative $\partial^{d}_{u/v}\mathbf{H}_{f}$ changes into $\partial
^{d}_{v/u}\mathbf{H}_{f}$, and the coefficient $(-1)^{n}$ appears. Hence, we
may group all the terms obtained after the transformation of the expression
$(\partial_{u}-\partial_{v})^{n}\mathbf{H}_{f}(u,v)$ into the pairs of the
form
\begin{equation}
\label{dtHGeneralTermPairs}\bigl(\partial_{u}^{d}+(-1)^{n}\partial^{d}%
_{v}\bigr)\mathbf{H}_{f}(u,v) \cdot\prod_{i=1}^{\ell}q_{f}^{(n_{i})}(u+v).
\end{equation}

Now we derive the recurrent relation for the derivative $(\partial
_{u}-\partial_{v})^{n}\mathbf{H}_{f}$. We may parametrize each term of the
form \eqref{dtHGeneralTermPairs} by the following parameters.

\begin{enumerate}
\item A number $\ell$ of the factors $q_{f}^{(n_{i})}(u+v)$. Since we need two
derivatives to obtain one factor $q(u+v)$ with the use of \eqref{GoursatTh1},
the number $\ell$ satisfies
\begin{equation}
\label{ParamEllIneq}0\le\ell\le\frac n2.
\end{equation}

\item A number $d$ of derivatives of $\mathbf{H}_{f}$ satisfying
\begin{equation}
\label{ParamDIneq}0\le d\le n-2\ell.
\end{equation}

\item An ordered sequence of non-negative numbers $n_{1}\le n_{2}\le\ldots\le
n_{\ell}$ (orders of derivatives of the factors $q_{f}(u+v)$) satisfying
\begin{equation}
\label{ParamNiIneq}n_{1}+n_{2}+\ldots+n_{\ell}+2\ell+d = n.
\end{equation}
\end{enumerate}

We denote by $(-1)^{\ell}S^{n}_{\ell; d; (n_{1},\ldots,n_{\ell})}$ the
coefficient at the term \eqref{dtHGeneralTermPairs} in the final
representation of $(\partial_{u}-\partial_{v})^{n}\mathbf{H}_{f}$ after the
described transformation. The factor $(-1)^{\ell}$ is used to make the number
$S^{n}_{\ell; d; (n_{1},\ldots,n_{\ell})}$ non-negative. We assume that
$S^{n}_{\ell; d; (n_{1},\ldots,n_{\ell})} = 0$ whenever either of the
conditions \eqref{ParamEllIneq}--\eqref{ParamNiIneq} fails.

\begin{lemma}
The number of different lists of parameters $\ell; d; (n_{1},\ldots,n_{\ell})$
satisfying conditions \eqref{ParamEllIneq}--\eqref{ParamNiIneq} for a fixed
$n$ coincides with the number $p(n)$ of partitions of $n$, i.e., the number of
ways $n$ may be represented as a sum of ordered positive integer terms, see,
e.g., \cite{Charal2002}, \cite{Comtet1974}. The number $p(n)$ is also known as
the number of different Young or Ferrers diagrams.
\end{lemma}

\begin{proof}
Consider a list of parameters $\ell;d;(n_{1},\ldots,n_{\ell})$ and put into
correspondence to this list a partition $\underbrace{1,1,\ldots,1}_{d},
n_{1}+2,n_{2}+2,\ldots, n_{\ell}+2$. Due to the condition \eqref{ParamNiIneq}
this partition is a partition of $n$. It is easy to see that the described
correspondence is one-to-one.
\end{proof}

Note that $(\partial_{u}-\partial_{v})q_{f}(u+v)=0$, hence in the case $d=0$
we have
\begin{multline}
\label{ReccurStepD0}(\partial_{u}-\partial_{v}) \biggl(\bigl(\mathbf{H}%
_{f}(u,v)+(-1)^{n}\mathbf{H}_{f}(u,v)\bigr)\prod_{i=1}^{\ell}q_{f}^{(n_{i}%
)}(u+v)\biggr)\\
=\bigl(1+(-1)^{n}\bigr)\bigl(\partial_{u}+(-1)^{n+1}\partial_{v}%
\bigr)\mathbf{H}_{f}(u,v)\cdot\prod_{i=1}^{\ell}q_{f}^{(n_{i})}(u+v),
\end{multline}
and in the case $d\ne0$ we have
\begin{multline}
\label{ReccurStepD}(\partial_{u}-\partial_{v}) \biggl(\bigl(\partial^{d}%
_{u}+(-1)^{n}\partial^{d}_{v}\bigr)\mathbf{H}_{f}(u,v)\prod_{i=1}^{\ell}%
q_{f}^{(n_{i})}(u+v)\biggr)\\
\displaybreak[2]
=\Bigl(\bigl(\partial^{d+1}_{u}+(-1)^{n+1}\partial^{d+1}_{v}\bigr)\mathbf{H}%
_{f}(u,v) -\bigl(\partial^{d-1}_{u}+(-1)^{n+1}\partial^{d-1}_{v}%
\bigr)\bigl(q_{f}(u+v)\mathbf{H}_{f}(u,v)\bigr)\Bigr)\prod_{i=1}^{\ell}%
q_{f}^{(n_{i})}(u+v)\\
\displaybreak[2]
=\bigl(\partial^{d+1}_{u}+(-1)^{n+1}\partial^{d+1}_{v}\bigr)\mathbf{H}%
_{f}(u,v)\cdot\prod_{i=1}^{\ell}q_{f}^{(n_{i})}(u+v)\\
-\sum_{k=0}^{d-1} \binom{d-1}{k} q_{f}^{(k)}(u+v)\bigl(\partial^{d-1-k}%
_{u}+(-1)^{n+1}\partial^{d-1-k}_{v}\bigr)\mathbf{H}_{f}(u,v)\cdot\prod
_{i=1}^{\ell}q_{f}^{(n_{i})}(u+v).
\end{multline}
Also note that the coefficient at the term $(\partial_{u}^{n}+(-1)^{n}%
\partial^{n}_{v})\mathbf{H}_{f}(u,v)$ is equal to $1/2$ for $n=0$ and $1$ for
$n\ge1$, hence $S^{0}_{0;0;()}=1/2$ and $S^{n}_{0; n; ()}=1$ for $n\ge1$. From
the relations \eqref{ReccurStepD0} and \eqref{ReccurStepD} and the latter
initial condition we obtain the following statement.

\begin{proposition}
\label{Prop S reccurrence} The coefficients $S^{n+1}_{\ell; d;(n_{1}%
,\ldots,n_{\ell})}$ with the lists of parameters $\ell; d;(n_{1}%
,\ldots,n_{\ell})$ satisfying conditions
\eqref{ParamEllIneq}--\eqref{ParamNiIneq} can be calculated by means of the
following recurrent relations
\begin{equation}
\label{S0reccurr}%
\begin{split}
S^{n+1}_{\ell; d;(n_{1},\ldots,n_{\ell})} =  &  \bigl(1+(-1)^{n}%
\bigr)S^{n}_{\ell; d-1;(n_{1},\ldots,n_{\ell})}\\
&  +\sum_{n_{k}\in\{n_{1},\ldots,n_{\ell}\}} \binom{d+n_{k}}{n_{k}}S^{n}%
_{\ell-1; d+n_{k}+1;(n_{1},\ldots,n_{k-1},n_{k+1},\ldots,n_{\ell})}%
\qquad\text{if } d=1,
\end{split}
\end{equation}
\begin{equation}
\label{Sreccurr}
\begin{split}S^{n+1}_{\ell; d;(n_{1},\ldots,n_{\ell})} = & S^{n}_{\ell;
d-1;(n_{1},\ldots,n_{\ell})}\\
& +\sum_{n_{k}\in\{n_{1},\ldots,n_{\ell}\}}
\binom{d+n_{k}}{n_{k}}S^{n}_{\ell-1; d+n_{k}+1;(n_{1},\ldots,n_{k-1}%
,n_{k+1},\ldots,n_{\ell})}\qquad\text{if } d\ne1,
\end{split}
\end{equation}
where $\sum_{n_{k}\in\{n_{1},\ldots,n_{\ell}\}}$ means that only different
values of $n_{i}$ are used in the sum, with the initial conditions
\begin{equation}
\label{SreccurrInit}S^{0}_{0;0;()}=1/2,\qquad S^{n}_{0;n;()} = 1\quad
\text{for}\ n\ge1.
\end{equation}
\end{proposition}

The following proposition presents a direct formula for the coefficients
$S^{n}_{\ell; d; (n_{1},\ldots,n_{\ell})}$ not involving recurrent relations.

\begin{proposition}
\label{Prop S direct} The following formula holds for the coefficient
$S^{n}_{\ell; d;(n_{1},\ldots,n_{\ell})}$, where $\ell\ge1$ and the list of
parameters $\ell; d;(n_{1},\ldots,n_{\ell})$ satisfies conditions
\eqref{ParamEllIneq}--\eqref{ParamNiIneq}:
\begin{multline}
\label{Sdirect}S^{n}_{\ell; d;(n_{1},\ldots,n_{\ell})} = \sum_{(\sigma
_{1},\ldots,\sigma_{\ell})\in\mathfrak{S}(n_{1},\ldots,n_{\ell})} \sum
_{d_{1}=0}^{d} \sum_{d_{2}=0}^{d_{1}+\sigma_{1}+1}\cdots\sum_{d_{\ell}%
=0}^{d_{\ell-1}+\sigma_{\ell-1}+1}\\
\prod_{i=1}^{\ell}\binom{d_{i}+\sigma_{i}}{\sigma_{i}}\bigl(1+(-1)^{\sigma
_{i}+\ldots+\sigma_{\ell}}\delta(d_{i})\bigl(1-\delta(d)\delta
(i-1)\bigr)\bigr),
\end{multline}
where $(\sigma_{1},\ldots,\sigma_{\ell})\in\mathfrak{S}(n_{1},\ldots,n_{\ell
})$ means that the sum is taken over all distinct permutations of the numbers
$n_{1},\ldots,n_{\ell}$, and $\delta(t)=1$ if $t=0$ and $\delta(t)=0$ otherwise.
\end{proposition}

\begin{proof}
To arrive from the term $S^{n}_{\ell;d;(n_{1},\ldots,n_{\ell})}$ to a term
$S^{\tilde n}_{0;\tilde n;()}$ using recurrent relations \eqref{S0reccurr} and
\eqref{Sreccurr} we have to perform exactly $\ell$ times the step of removing
one of the $n_{k}$'s from the list $(n_{1},\ldots,n_{\ell})$. Denote by
$d_{1},\ldots,d_{\ell}$ the values of the parameter $d$ on each of these steps
and by $\sigma_{1},\ldots,\sigma_{\ell}$ the number $n_{k}$ removed from the
list $(n_{1},\ldots,n_{\ell})$ on the corresponding step $1,\ldots,\ell$. Then
the numbers $d_{1},\ldots,d_{\ell}$ satisfy the following inequalities
\begin{align*}
0 \le\,  &  d_{1}\le d,\\
0 \le\,  &  d_{2}\le1+d_{1}+\sigma_{1},\\
&  \, \vdots\\
0\le\,  &  d_{\ell}\le1+d_{\ell-1}+\sigma_{\ell-1}.
\end{align*}
Each distinct permutation of the numbers $n_{1},\ldots,n_{\ell}$ and each list
of parameters $d_{1},\ldots,d_{\ell}$ give us a different way of getting from
the term $S^{n}_{\ell;d;(n_{1},\ldots,n_{\ell})}$ to the term $S^{\tilde
n}_{0;\tilde n;()}$. Note that the factor $1+(-1)^{n^{\prime}}$ for some
$n^{\prime}\le n$ appears in \eqref{S0reccurr} only in the case when we cross
the coefficient $S^{n^{\prime}}_{\ell^{\prime};0;(n^{\prime}_{1}%
,\ldots,n^{\prime}_{\ell^{\prime}})}$ having $d^{\prime}=0$. Note also that if
$d=0$ initially, then $d_{1}=0$ and only \eqref{Sreccurr} is applicable, hence
the factor $1+(-1)^{n}$ does not appear in this case. So the factor
$1+(-1)^{n^{\prime}}$ appears only in one of the following cases: 1) $d\ne0$
and $d_{1}=0$, 2) $d_{i}=0$, $i\ge2$. It follows from condition
\eqref{ParamNiIneq} that in each of these cases $n^{\prime}=2\ell^{\prime
}+n_{1}^{\prime}+\ldots+n^{\prime}_{\ell^{\prime}}$ therefore $(-1)^{n^{\prime
}} = (-1)^{n_{1}^{\prime}+\ldots+n^{\prime}_{\ell^{\prime}}}$. Hence we
conclude that the following relation holds
\begin{multline*}
S^{n}_{\ell; d;(n_{1},\ldots,n_{\ell})} = \sum_{(\sigma_{1},\ldots
,\sigma_{\ell})\in\mathfrak{S}(n_{1},\ldots,n_{\ell})} \sum_{d_{1}=0}^{d}
\binom{d_{1}+\sigma_{1}}{\sigma_{1}} \bigl(1+(-1)^{\sigma_{1}+\ldots
+\sigma_{\ell}}\cdot\delta(d_{1})(1-\delta(d))\bigr)\\
\times\sum_{d_{2}=0}^{d_{1}+\sigma_{1}+1}\binom{d_{2}+\sigma_{2}}{\sigma_{2}}
\bigl(1+(-1)^{\sigma_{2}+\ldots+\sigma_{\ell}}\delta(d_{2})\bigr)\cdots
\sum_{d_{\ell}=0}^{d_{\ell-1}+\sigma_{\ell-1}+1}\binom{d_{\ell}+\sigma_{\ell}%
}{\sigma_{\ell}} \bigl(1+(-1)^{\sigma_{\ell}}\delta(d_{\ell})\bigr),
\end{multline*}
where the first sum is taken over all distinct permutations of numbers
$n_{1},\ldots,n_{\ell}$, and the function $\delta(\cdot)$ is used to the
encode mentioned above conditions 1) and 2). Finally note that the conditions
1) and 2) may be jointly represented by the expression $\delta(d_{i}%
)\bigl(1-\delta(d)\delta(i-1)\bigr)$.
\end{proof}

As a result we obtain the following representation for the derivative
$(\partial_{u}-\partial_{v})^{n}\mathbf{H}_{f}(u,v)$, $n\ge1$.

\begin{proposition}
\label{Prop Repres dudvH} Let $n\ge1$ and $q_{f}\in C^{(n-1)}[-b,b]$. Then
\begin{multline*}
(\partial_{u}-\partial_{v})^{n}\mathbf{H}_{f}(u,v) = \bigl(\partial^{n}%
_{u}+(-1)^{n}\partial_{v}^{n}\bigr)\mathbf{H}_{f}(u,v)\\
+\sum_{\ell=1}^{[n/2]} (-1)^{\ell}\sum_{d=0}^{n-2\ell}\bigl(\partial_{u}%
^{d}+(-1)^{n}\partial^{d}_{v}\bigr)\mathbf{H}_{f}(u,v) \sum_{\substack{n_{1}%
+\ldots+n_{\ell}=n-2\ell-d,\\0\le n_{1}\le\ldots\le n_{\ell}}} S^{n}%
_{\ell;d;(n_{1},\ldots,n_{\ell})} \prod_{i=1}^{\ell}q_{f}^{(n_{i})}(u+v),
\end{multline*}
where the coefficients $S^{n}_{\ell;d;(n_{1},\ldots,n_{\ell})}$ are given by
Proposition \ref{Prop S reccurrence} or Proposition \ref{Prop S direct} and do
not depend on $q_{f}$.
\end{proposition}

The following proposition is a corollary of \eqref{dtKdudvH}, \eqref{duH(0)},
\eqref{dvH(0)} and Proposition \ref{Prop Repres dudvH}.

\begin{proposition}
Let $n\ge1$ and $q_{f}\in C^{(n-1)}[-b,b]$. Then
\begin{multline}
\label{dtK0}\partial_{t}^{n}\mathbf{K}_{f}(0,0) = \frac1{2^{n+1}}
\biggl( q_{f}^{(n-1)}(0) + \sum_{\ell=1}^{[n/2]} (-1)^{\ell}\sum
_{d=0}^{n-2\ell} \bigl(1+(-1)^{n}\delta(d)\bigr)q_{f}^{(d-1)}(0)\\
\times\sum_{\substack{n_{1}+\ldots+n_{\ell}=n-d-2\ell,\\0\le n_{1}\le\ldots\le
n_{\ell}}}S^{n}_{\ell;d;(n_{1},\ldots,n_{\ell})}\prod_{i=1}^{\ell}%
q_{f}^{(n_{i})}(0)\biggr),
\end{multline}
where we set $q_{f}^{(-1)}(0):=h=f^{\prime}(0)$, the coefficients $S^{n}%
_{\ell;d;(n_{1},\ldots,n_{\ell})}$ are given by Proposition
\ref{Prop S reccurrence} or Proposition \ref{Prop S direct} and do not depend
on $q_{f}$, and $\delta(d)=1$ if $d=0$ and $\delta(d)=0$ otherwise.

Formula \eqref{dtK0} also holds for $n=0$ for any $q_{f}\in C[-b,b]$.
\end{proposition}

As a consequence of Theorems \ref{Th W series}, \ref{Th Taylor Vekua} and
\ref{Th K Vekua} we obtain the following representation of the integral kernel
$\mathbf{K}_{f}$.

\begin{theorem}
\label{Th Formula K} Let $q_{f}\in C^{\infty}[-b,b]$ be a complex-valued
function and $f$ be a particular solution of \eqref{SLhom} such that $f\neq0$
on $[-b,b]$ and normalized as $f(0)=1$. Denote $h:=f^{\prime}(0)\in\mathbb{C}%
$. Suppose that the functions $g_{1}(x):=\frac12\bigl(\mathbf{K}%
_{f}(x,x)+\mathbf{K}_{f}(x,-x)\bigr)=\frac h2+\frac14\int_{0}^{x}
q_{f}(s)\,ds$ and $g_{2}(x):=\frac12\bigl(\mathbf{K}_{f}(x,x)-\mathbf{K}%
_{f}(x,-x)\bigr)=\frac14\int_{0}^{x} q_{f}(s)\,ds$ admit uniformly convergent
on $[-b,b]$ series expansions
\[
g_{1}(x) = c_{0} u_{0}(x,x) + \sum_{n=1}^{\infty}c_{n} u_{2n-1}(x,x)
\]
and
\[
g_{2}(x) = \sum_{n=1}^{\infty}b_{n} u_{2n}(x,x).
\]
Then the coefficients $\{c_{n}\}_{n\ge0}$ and $\{b_{n}\}_{n\ge1}$ may be found
by the formulas
\[
c_{n} =\frac{\partial_{t}^{n}\mathbf{K}_{f}(0,0)}{n!},\qquad b_{n} =
-\frac{\partial_{t}^{n}\mathbf{K}_{1/f}(0,0)}{n!},\quad\text{if } n\ \text{is
even}%
\]
and
\[
c_{n} =-\frac{\partial_{t}^{n}\mathbf{K}_{1/f}(0,0)}{n!},\qquad b_{n} =
\frac{\partial_{t}^{n}\mathbf{K}_{f}(0,0)}{n!},\quad\text{if } n\ \text{is
odd},
\]
where the derivatives $\partial_{t}^{n}\mathbf{K}_{f}(0,0)$ and $\partial
_{t}^{n}\mathbf{K}_{1/f}(0,0)$ are given by \eqref{dtK0} for the potentials
$q_{f}$ and $q_{1/f}=-q_{f}+2\bigl(f^{\prime}/f\bigr)^{2}$, respectively.
Also, for any $(x,t)\in\overline{\mathbf{R}}$
\begin{equation}
\label{FormulaK}\mathbf{K}_{f}(x,t) = c_{0}u_{0}(x,t)+\sum_{n=1}^{\infty
}\left(  c_{n}u_{2n-1}(x,t)+b_{n}u_{2n}(x,t)\right)  ,
\end{equation}
and the series converges uniformly in $\overline{\mathbf{R}}$.
\end{theorem}

\begin{example}
For the potentials $q_{f}$ and $q_{1/f}$ from Example \ref{ModelExample} we
have
\[
q_{f}^{(-1)}(0) := f^{\prime}(0) = 1,\qquad q_{f}^{(n)}(0) = 0,\quad n\ge0,
\]
and
\[
q_{1/f}^{(-1)}(0) :=(1/f)^{\prime}(0)= -1,\qquad q_{1/f}^{(n)}(0) =
2(-1)^{n}(n+1)!,\quad n\ge0.
\]
By Proposition \ref{Prop S reccurrence} the first coefficients $S^{n}%
_{\ell;d;(n_{1},\ldots,n_{\ell})}$ for $n\le6$ have the following values
\begin{gather*}
S^{1}_{0;1;()}=1;\\
S^{2}_{0;2;()}=1,\quad S^{2}_{1;0;(0)}=1;\\
S^{3}_{0;3;()}=1;\quad S^{3}_{1;0;(1)}=1,\quad S^{3}_{1;1;(0)}=3;\\
S^{4}_{0;4;()}=1;\quad S^{4}_{1;0;(2)}=1,\quad S^{4}_{1;1;(1)}=2;\quad
S^{4}_{1;2;(0)}=4,\quad S^{4}_{2;0;(0,0)}=3;\\
S^{5}_{0;5;()}=1;\ S^{5}_{1;0;(3)}=1,\ S^{5}_{1;1;(2)}=5;\ S^{5}%
_{1;2;(1)}=5,\ S^{5}_{1;3;(0)}=5,\ S^{5}_{2;0;(0,1)}=6,\ S^{5}_{2;1;(0,0)}%
=10;\\
S^{6}_{0;6;()}=1;\quad S^{6}_{1;0;(4)}=1,\quad S^{6}_{1;1;(3)}=4;\quad
S^{6}_{1;2;(2)}=11,\quad S^{6}_{1;3;(1)}=9,\quad S^{6}_{1;4;(0)}=6,\\
S^{6}_{2;0;(0,2)}=10,\quad S^{6}_{2;0;(1,1)}=5,\quad S^{6}_{2;1;(0,1)}%
=15,\quad S^{6}_{2;2;(0,0)}=15,\quad S^{6}_{3;0;(0,0,0)}=10.
\end{gather*}
We can check that formula \eqref{dtK0} works, e.g.,
\begin{multline*}
\displaybreak[2] \partial_{t}^{5}\mathbf{K}_{1/f}(0,0) = \frac1{64}\Bigl(240+
(-1)^{1}\bigl(1\cdot(1+(-1)^{5})\cdot48 + 5\cdot24 + 5\cdot16+5\cdot24\bigr)\\
+(-1)^{2}\bigl(6\cdot(1+(-1)^{5})\cdot8+10\cdot8\bigr)\Bigr)=0,
\end{multline*}
and
\begin{multline*}
\displaybreak[2] \partial_{t}^{6}\mathbf{K}_{1/f}(0,0) = \frac1{128}%
\Bigl(-1440+ (-1)^{1}\bigl(1\cdot(1+(-1)^{6})\cdot(-240) + 4\cdot(-96) +
11\cdot(-48)\\
\displaybreak[2]
+9\cdot(-48)+6\cdot(-96)\bigr)
 +(-1)^{2}\bigl(10\cdot(1+(-1)^{6})\cdot(-24)+5\cdot
(1+(-1)^{6})\cdot(-16)\\
+15\cdot(-16)+15\cdot(-16)\bigr)+(-1)^{3}\bigl(10\cdot(1+(-1)^{6})\cdot(-8)\bigr)\Bigr)=0.
\end{multline*}
\end{example}

Now consider an example when all the coefficients $b_{n}$, $c_{n}$ from
\eqref{FormulaK} may be easily calculated explicitly from \eqref{Sdirect} and \eqref{dtK0}.

\begin{example}
Consider $q_{f}=c^{2}$ for some constant $c$ taking the function $f = e^{cx}$
as a non-vanishing solution of the equation $f^{\prime\prime}-q_{f} f=0$. Then
the Darboux transformed potential has the form $q_{1/f} = 2(f^{\prime}%
/f)^{2}-q_{f} = c^{2}$.

Consider the function $\mathbf{H}_{f}$ from \eqref{dtKdudvH}. Equation
\eqref{GoursatTh1} for $q_{f}=c^{2}$ reads as $\partial_{u}\partial
_{v}\mathbf{H}_{f} = c^{2}\mathbf{H}_{f}$, hence
\begin{equation}
\label{ExampleHconst}%
\begin{split}
(\partial_{u}-\partial_{v})^{n}\mathbf{H}_{f}(u,v)  &  = \sum_{k=0}^{n}
(-1)^{n-k}\binom{n}{k}\partial_{u}^{k}\partial_{v} ^{n-k}\mathbf{H}_{f}(u,v)\\
&  = (\partial_{u}^{n} + (-1)^{n}\partial_{v}^{n})\mathbf{H}_{f}(u,v) \\
& +\sum_{k=1}^{[n/2]-1} (-1)^{n-k} c^{2k} \binom{n}{k}\partial_{v}^{n-2k}%
\mathbf{H}_{f}(u,v)\\
&  + \sum_{k=[n/2]}^{n-1} (-1)^{n-k} c^{2(n-k)} \binom{n}{k}\partial
_{u}^{2k-n}\mathbf{H}_{f}(u,v).
\end{split}
\end{equation}
It follows from \eqref{duH(0)} and \eqref{dvH(0)} that at the point $u=v=0$
every term in \eqref{ExampleHconst} involving a derivative of $\mathbf{H}_{f}$
of order at least 2 with respect to $u$ or at least 1 with respect to $v$
equals zero. Hence from \eqref{dtKdudvH} for any even $n=2m$ we obtain
\begin{equation}
\label{ExampleConstK2m}\partial_{t}^{2m}\mathbf{K}_{f}(0,0)=\frac1{2^{2m}%
}(-1)^{m} c^{2m}\binom{2m}{m}\mathbf{H}_{f}(0,0) = \frac{(-1)^{m} c^{2m+1}%
}{2^{2m+1}}\binom{2m}{m},
\end{equation}
and for any odd $n=2m+1$ we obtain
\begin{equation}
\label{ExampleConstK2m+1}\partial_{t}^{2m+1}\mathbf{K}_{f}(0,0)=\frac
1{2^{2m+1}}(-1)^{m} c^{2m}\binom{2m+1}{m}\partial_{u}\mathbf{H}_{f}(0,0) =
\frac{(-1)^{m} c^{2m+2}}{2^{2m+2}}\binom{2m+1}{m}.
\end{equation}
Similarly for the kernel $\mathbf{K}_{1/f}$ we have
\[
\partial_{t}^{2m}\mathbf{K}_{1/f}(0,0)=-\frac{(-1)^{m} c^{2m+1}}{2^{2m+1}%
}\binom{2m}{m},\qquad\partial_{t}^{2m+1}\mathbf{K}_{1/f}(0,0)= \frac{(-1)^{m}
c^{2m+2}}{2^{2m+2}}\binom{2m+1}{m}.
\]
Hence by Theorem \ref{Th Formula K} the expansion coefficients have the form
\begin{align*}
c_{2m}  &  = \frac{(-1)^{m} c^{2m+1}}{2^{2m+1} (m!)^{2}}, & b_{2m}  &  =
\frac{(-1)^{m} c^{2m+1}}{2^{2m+1} (m!)^{2}},\\
c_{2m+1}  &  = \frac{(-1)^{m+1} c^{2m+2}}{2^{2m+2} m! (m+1)!}, & b_{2m+1}  &
= \frac{(-1)^{m} c^{2m+2}}{2^{2m+2} m! (m+1)!}.
\end{align*}
Now we show that it is possible to obtain the same coefficients from
\eqref{Sdirect} and \eqref{dtK0}. Note that any term in \eqref{dtK0} involving
a positive order derivative of $q_{f}$ is equal to zero. Hence $n_{1}%
=n_{2}=\ldots=n_{\ell}=0$ and $n=2\ell+d$. Moreover, $d\le1$ and is equal
either to 0 or 1 depending on the parity of $n$. So we have to calculate only
one coefficient, $S^{2m}_{m;0;(0,\ldots,0)}$ for $n=2m$ and $S^{2m+1}%
_{m;1;(0,\ldots,0)}$ for $n=2m+1$. Formula \eqref{Sdirect} for the coefficient
$S^{2m}_{m;0;(0,\ldots,0)}$ takes the form
\[
S^{2m}_{m;0;(0,\ldots,0)} = \sum_{d_{2}=0}^{1} \bigl(1+\delta(d_{2}%
)\bigr)\sum_{d_{3}=0}^{d_{2}+1}\bigl(1+\delta(d_{3})\bigr)\cdots\sum_{d_{m}%
=0}^{d_{m-1}+1}\bigl(1+\delta(d_{m})\bigr).
\]
We claim that
\begin{equation}
\label{ExampleConstCoeff}\sum_{d_{m-k+1}=0}^{d_{m-k}+1} \bigl(1+\delta
(d_{m-k+1})\bigr) \cdots\sum_{d_{m}=0}^{d_{m-1}+1}\bigl(1+\delta
(d_{m})\bigr) = \binom{2k+1+d_{m-k}}{k}.
\end{equation}
Indeed, for $k=1$ we have $\sum_{d_{m}=0}^{d_{m-1}+1}\bigl(1+\delta
(d_{m})\bigr) = d_{m-1}+3 = \binom{3+d_{m-1}}{1}$. Assume that the claim is
correct for some $k$, then for $k+1$ we obtain
\begin{multline*}
\sum_{d_{m-k}=0}^{d_{m-k-1}+1}\bigl(1+\delta(d_{m-k})\bigr)\cdot
\binom{2k+1+d_{m-k}}{k} \\= 2\binom{2k+1}{k}+\binom{2k+2}{k}+\ldots
+\binom{2k+2+d_{m-k-1}}{k},
\end{multline*}
and the induction step follows from the equality $2\binom{2k+1}{k}%
+\binom{2k+2}{k} = \binom{2k+3}{k+1}$ and the well-known combinatorial
identity $\binom{n+1}{k+1}=\binom{n}{k}+\binom{n-1}{k}+\ldots+\binom{n-s}%
{k}+\binom{n-s}{k+1}$, see, e.g., \cite{Comtet1974}. For $k=m-1$ we obtain
from \eqref{ExampleConstCoeff} that $S^{2m}_{m;0;(0,\ldots,0)} = \binom
{2m-1}{m-1}$ and $\partial_{t}^{2m}\mathbf{K}_{f}(0,0) = \frac{(-1)^{m}%
}{2^{2m+1}}2c\binom{2m-1}{m-1}(c^{2})^{m} = \frac{(-1)^{m} c^{2m+1}}{2^{2m+1}%
}\binom{2m}{m}$.

Similarly, for $n=2m+1$ we obtain from \eqref{ExampleConstCoeff} for $k=m$
that $S^{2m+1}_{m;1;(0,\ldots,0)} = \binom{2m+1}{m}$ and $\partial_{t}%
^{2m+1}\mathbf{K}_{f}(0,0) = \frac{(-1)^{m}}{2^{2m+2}}c^{2}\binom{2m+1}%
{m+1}(c^{2})^{m} = \frac{(-1)^{m} c^{2m+2}}{2^{2m+2}}\binom{2m+1}{m}$, exactly
as in \eqref{ExampleConstK2m} and \eqref{ExampleConstK2m+1}.
\end{example}

In the following example we calculate the expansion coefficients $b_{n}$ and
$c_{n}$, $n\le22$, for a reflectionless potential in the one-dimensional
quantum scattering theory, $q_{\sech}(x) = 1-2\sech^{2} x$, see, e.g.,
\cite{KrT2012}, \cite{Lamb}. For this potential to the difference with from
the previous examples neither the higher derivatives of $q_{f}$ nor the higher
derivatives with respect to $t$ of the transmutation kernel do not vanish identically.

\begin{example}
\label{ExampleCosh} The potential $q_{\sech}(x) = 1-2\sech^{2} x$ can be
obtained as a result of the Darboux transformation of the equation
$u^{\prime\prime}=u$ with the potential $q_{\cosh}\equiv1$ with respect to the
solution $f(x)=\cosh x$. The transmutation operator for the operator
$A_{\cosh}=\partial_{x}^{2}-1$ was calculated in \cite[Example 3]{CKT}. Its
kernel is given by the expression
\[
\mathbf{K}_{\cosh}(x,t)=\frac{1}{2}\frac{\sqrt{x^{2}-t^{2}}I_{1}(\sqrt
{x^{2}-t^{2}})}{x-t},
\]
where $I_{1}$ is the modified Bessel function of the first kind. Even though
in \cite{KrT2012} a formula for the integral kernel $\mathbf{K}_{\sech}$ was
presented, it is not well suited for calculating higher order derivatives with
respect to $t$. Using \eqref{K2Vekua} we obtain another representation with
the help of Maple 12 software
\begin{multline}
\label{Ksech}\mathbf{K}_{\sech}(x,t) =\frac12\biggl(I_{1}(x) - I_{0}(x)\tanh x
+\tanh x\int_{0}^{t}\frac{\sqrt{x^{2}-s^{2}} I_{1}(\sqrt{x^{2}-s^{2}})}%
{x-s}ds\\
+\int_{0}^{t} \frac{(xs-x^{2})I_{0}(\sqrt{x^{2}-s^{2}}) +\sqrt{x^{2}-s^{2}}
I_{1}(\sqrt{x^{2}-s^{2}})}{(x-s)^{2}}ds \biggr).
\end{multline}
Despite the integral kernel $\mathbf{K}_{\sech}$ could not be evaluated
explicitly, it is possible to expand it into a Taylor series at the point
$t=0$. Expanding the kernels $\mathbf{K}_{\cosh}$ and $\mathbf{K}_{\sech}$
into corresponding Taylor series and taking a limit as $x\to0$ we obtain
(again with the help of Maple 12 software) that
\begin{multline*}
\mathbf{K}_{\cosh}(0, t) = \frac{1}{4}\,t - \frac{1}{32} \,t^{3} + \frac
{1}{768} \,t^{5} - \frac{1}{36864} \,t^{7} + \frac{ 1}{2949120} \,t^{9} -
\frac{1}{353894400} \,t^{ 11} \\
+ \frac{1}{59454259200} \,t^{13}
- \frac{1}{13317754060800} \,t^{15} + \frac{1}{3835513169510400} \,t^{17} \\
- \frac{1}{1380784741023744000} \,t^{19}
+ \frac{1}{607545286050447360000}\,t^{21} + o(t^{22})
\end{multline*}
and
\begin{multline*}
\mathbf{K}_{\sech}(0, t) =- \frac{1}{4} \,t + \frac{1}{96 } \,t^{3} - \frac
{1}{3840} \,t^{5} + \frac{1}{258048} \,t^{7} - \frac{1}{26542080} \,t^{9} +
\frac{1}{ 3892838400} \,t^{11} \\
- \frac{1}{772905369600} \,t^{13} + \frac{1}{199766310912000} \,t^{15} - \frac{1}{65203723881676800} \,t^{17} \\
+ \frac{1}{26234910079451136000} \,t^{19}
- \frac{1}{12758451007059394560000} \,t^{21} + o(t^{22})
\end{multline*}
According to Theorem \ref{Th Formula K} we obtain the following table of the
expansion coefficients for the integral kernel $\mathbf{K}_{\sech}$. Since all
even coefficients are equal to 0, we do not include them in the table. To save
the space, we represent all odd coefficients in the form of a fraction
$\frac{x_{n}}{2^{n+1}n!}$.

\begin{center}
\noindent%
\begin{tabular}
[c]{ccccccc}\hline
$n$ & $1$ & $3$ & $5$ & $7$ & $9$ & $11$\\\hline
$b_{n}$ & $-\frac{1}{2^{2} 1!}$ & $\frac{1}{2^{4} 3!}$ & $-\frac{2}{2^{6} 5!}$
& $\frac{5}{2^{8} 7!}$ & $-\frac{14}{2^{10} 9!}$ & $\frac{42}{2^{12} 11!}$\\
\hline
$c_{n}$ & $-\frac{1}{2^{2} 1!}$ & $\frac{3}{2^{4} 3!}$ & $-\frac{10}{2^{6}
5!}$ & $\frac{35}{2^{8} 7!}$ & $-\frac{126}{2^{10} 9!}$ & $\frac{462}{2^{12}
11!}$\\\hline
\end{tabular}

\noindent%
\begin{tabular}
[c]{cccccc}\hline
$n$ & $13$ & $15$ & $17$ & $19$ &
$21$\\\hline
$b_{n}$ & $-\frac{132}{2^{14} 13!}$ & $\frac{429}{2^{16} 15!}$ & $-\frac{1430}{2^{18}
17!}$ & $\frac{4862}{2^{20} 19!}$ & $-\frac{16796}{2^{22} 21!}$\\\hline
$c_{n}$ &  $-\frac{1716}{2^{14} 13!}$ & $\frac{6435}{2^{16} 15!}$ &
$-\frac{24310}{2^{18} 17!}$ & $\frac{92378}{2^{20} 19!}$ & $-\frac
{352716}{2^{22} 21!}$\\\hline
\end{tabular}
\end{center}

For $f = \cosh x$ we have $h=f^{\prime}(0)=0$. The derivatives at the point
$x=0$ of the potential $q_{1}\equiv1$ are equal to $q_{1}^{(0)}(0)=1$,
$q_{1}^{(n)}(0)=0$, $n\ge1$, the non-zero values of the derivatives of the
potential $q(x) = 1-2\sech^{2}(x)$ at the same point are given in the
following table.

\begin{center}
\noindent%
\begin{tabular}
[c]{cccccccc}\hline
$n$ & $0$ & $2$ & $4$ & $6$ & $8$ & $10$ & $12$\\\hline
$q^{(n)}(0)$ & $-1$ & $4$ & $-32$ & $544$ & $-15872$ & $707584$ &
$-44736512$\\\hline
\end{tabular}

\noindent%
\begin{tabular}
[c]{ccccc}\hline
$n$ & $14$ & $16$ & $18$ & $20$\\\hline
$q^{(n)}(0)$ & $3807514624$ & $-419730685952$ & $58177770225664$ &
$-9902996106248192$\\\hline
\end{tabular}
\end{center}

The table below presents the values of the expression in parentheses appearing
in \eqref{dtK0}, obtained from the derivatives of the potentials presented
above with the coefficients $S^{n}_{\ell;d;(n_{1},\ldots,n_{\ell})}$
calculated by Proposition \ref{Prop S reccurrence}. All calculations were
performed with the help of Matlab 7 and the results agree with the ones
obtained from the explicit formulae for the integral kernels (see the table
for $b_{n}$ and $c_{n}$ above). The value $n=21$ is chosen only to keep the
size of the table reasonable. Our Matlab program calculated exact values of
the derivatives $\partial_{t}^{n}\mathbf{K}_{\cosh}(0,0)$ for $n\le30$ and
exact values of the derivatives $\partial_{t}^{n}\mathbf{K}_{\sech}(0,0)$ for
$n\le26$. To calculate the derivatives for larger values of $n$ one have to
use arbitrary precision arithmetic to overcome the limitation of the machine precision.

\begin{center}
\noindent%
\begin{tabular}
[c]{ccccccc}\hline
$n$ & $1$ & $3$ & $5$ & $7$ & $9$ & $11$ \\\hline
$2^{n+1}\partial_{t}^{n}\mathbf{K}_{\cosh}$ & $1$ & $-3$ & $10$ & $-35$ &
$126$ & $-462$\\\hline
$2^{n+1}\partial_{t}^{n}\mathbf{K}_{\sech}$ & $-1$ & $1$ & $-2$ & $5$ & $-14$
& $42$\\\hline
\end{tabular}

\noindent%
\begin{tabular}
[c]{cccccc}\hline
$n$ & $13$ & $15$ & $17$ & $19$ &
$21$\\\hline
$2^{n+1}\partial_{t}^{n}\mathbf{K}_{\cosh}$ & $1716$ & $-6435$ & $24310$ & $-92378$ & $352716$\\\hline
$2^{n+1}\partial_{t}^{n}\mathbf{K}_{\sech}$ & $-132$ & $429$ & $-1430$ & $4862$ & $-16796$\\\hline
\end{tabular}
\end{center}
\end{example}

\section{Approximate construction of integral kernels}\label{SectApproxKernel}

In this section we discuss two methods of approximate construction of integral
kernels of transmutation operators. The first method is based on the expansion
theorem and computation of the generalized Taylor coefficients by Theorem
\ref{Th Formula K} and formula \eqref{dtK0}. The second method is based on the
completeness result of hyperbolic formal powers and of generalized wave
polynomials (Proposition \ref{Prop RungeZ} and \cite[Theorem 22]{KKTT}) and on
approximating the Goursat data for the integral kernel by generalized wave
polynomials. However we present only the description of the main ideas of
these approximate methods and do not perform detailed analysis. Such analysis,
involving studying of approximative properties of the functions $\{u_{n}%
(x,x)\}$, convergency rate estimates and detailed comparison with existing
numerical techniques, e.g., successive approximation and the series expansions
from \cite{Boumenir2006}, goes beyond the scope of the present article. The
authors hope to perform such analysis and present the results shortly.

The first method is applicable in the case when we know derivatives of the
function $f\in C^{n}[-b,b]$ at the point $x=0$ of all orders up to some order
$n$, $n\geq2$, i.e.\ when we know coefficients in the Taylor formula
\[
f(x)=\sum_{k=0}^{n}f_{k}x^{k}+o(x^{k}).
\]
Suppose additionally that the function $f$ does not vanish on $[-b,b]$ and
normalized as $f(0)=1$, i.e.\ $f_{0}=1$. Then we can calculate Taylor
coefficients up to the order $n$ of the function $1/f$. For example, as it
follows from \cite[Theorem 11.7]{Charal2002} or directly from the Fa\`{a} di
Bruno formula \cite[Theorem 11.4]{Charal2002},
\[
\frac{1}{f(x)}=1+\sum_{k=1}^{n}\tilde{f}_{k}x^{k}+o(x^{k}),
\]
where
\begin{equation}
\tilde{f}_{k}=\sum_{\substack{m_{1}+2m_{2}+\ldots+km_{k}=k\\m_{1}\geq
0,\ldots,m_{k}\geq0}}(-1)^{m_{1}+\ldots+m_{k}}\binom{m_{1}+\ldots+m_{k}}%
{m_{1},\ldots,m_{k}}\prod_{j=1}^{k}f_{j}^{m_{j}}, \label{fInverseTaylor}%
\end{equation}
and $\binom{m_{1}+\ldots+m_{k}}{m_{1},\ldots,m_{k}}=\frac{(m_{1}+\ldots
+m_{k})!}{m_{1}!\cdot\ldots\cdot m_{k}!}$ are multinomial coefficients.

Having coefficients $f_{k}$ and $\tilde f_{k}$, $k\le n$, we easily calculate
the Taylor coefficients up to the order $n-2$ of the potentials $q_{f} =
f^{\prime\prime}/f$ and $q_{1/f} = f\cdot(1/f)^{\prime\prime}$. By
\eqref{dtK0} we obtain the values of the derivatives $\partial_{t}%
^{m}\mathbf{K}_{f}(0,0)$ and $\partial_{t}^{m}\mathbf{K}_{1/f}(0,0)$, $m\le
n-1$, and by Theorem \ref{Th Formula K} we find the coefficients $b_{k}$ and
$c_{k}$, $k\le n-1$ in the representation \eqref{FormulaK}. The obtained
truncated series is an approximation of the integral kernel $\mathbf{K}_{f}$.

If on the opposite the Taylor coefficients of the potential $q_{f}$ and the
value of the parameter $h=f^{\prime}(0)$ are known, the reconstruction of the
Taylor coefficients of the solution $f$ also presents no difficulty.

\begin{example}
\label{ExampleCosh2} Consider the function $f=\cosh x$ as in Example
\ref{ExampleCosh}. One can easily verify that the first Taylor coefficients
for the function $1/f$ calculated by \eqref{fInverseTaylor} coincide with the
coefficients in the known expansion
\[
\frac{1}{\cosh x}=\sech x=1-\frac{1}{2}x^{2}+\frac{5}{24}x^{4}-\frac{61}%
{720}x^{6}+\frac{277}{8064}x^{8}+O(x^{10}),
\]
and the calculated Taylor coefficients for the potentials $q_{f}$ and
$q_{1/f}$ coincide with the values presented in Example \ref{ExampleCosh}.%

\begin{figure}[tbh]
\centering
\includegraphics[
natheight=2.303800in,
natwidth=3.135700in,
height=2.93in,
width=3.975in
]%
{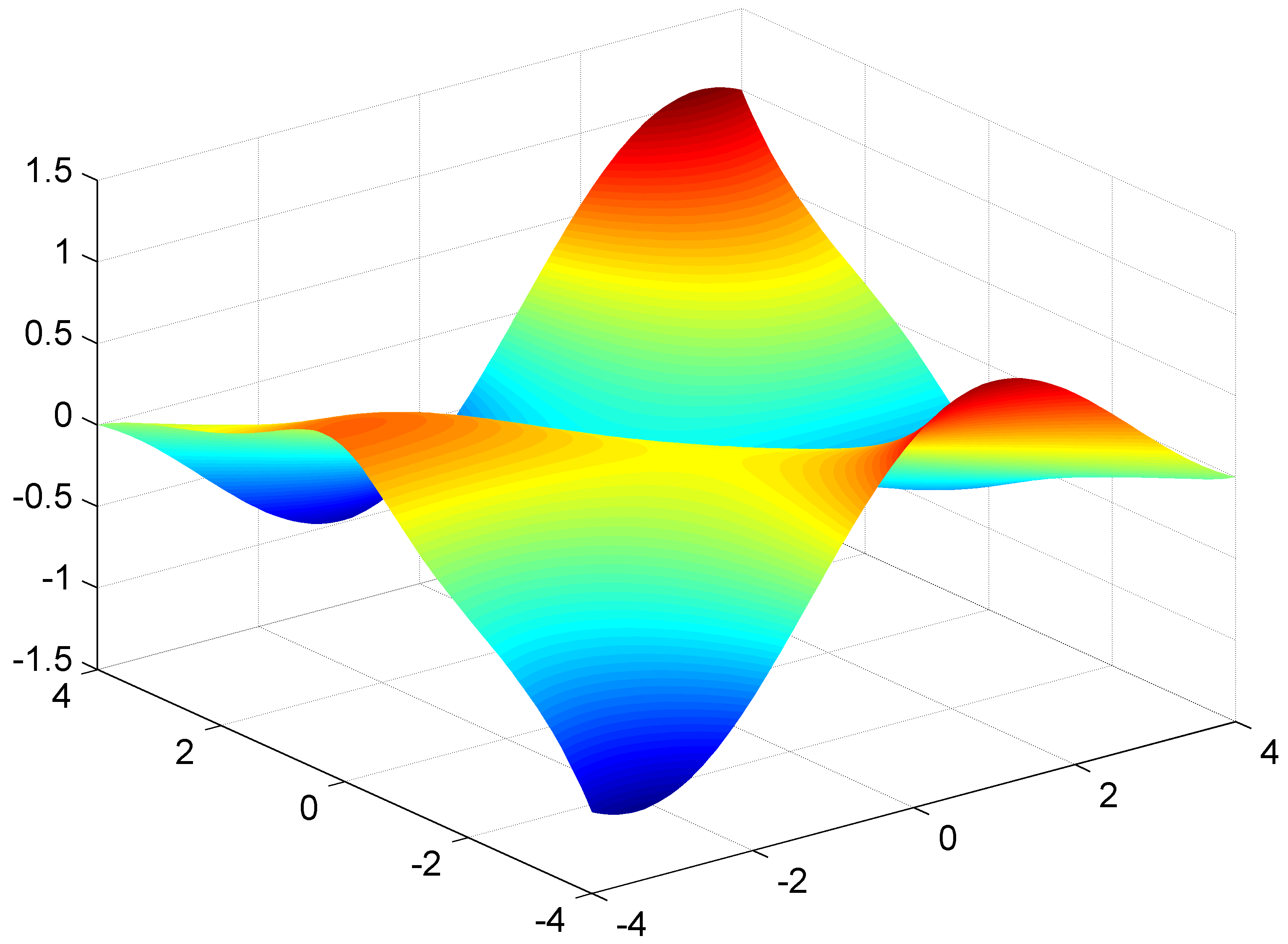}
\caption{The integral kernel $K_{\sech}$ in the domain
$\overline{\mathbf{R}}$  given by $b=4$.}
\end{figure}

We compare the exact integral kernel $\mathbf{K}_{\sech}$ given by
\eqref{Ksech} with its approximations based on the generalized Taylor
coefficients calculated in Example \ref{ExampleCosh}. To illustrate the
dependence of the approximation error on the domain $\overline{\mathbf{R}}$
and on the number of coefficients used, we present the resulting approximation
error for 3 different domains corresponding to values $b=1$, $b=2$ and $b=4$,
and for different number of generalized Taylor coefficients. The functions
$\varphi_{k}$, defined by \eqref{phik} and necessary for constructing the
generalized wave polynomials, are calculated using two Matlab routines from
the Spline Toolbox: on each step the integrand is approximated by a spline
with 5000 knots using the command \texttt{spapi} and then it is integrated
using \texttt{fnint}. The resulting kernel $\mathbf{K}_{\sech}$ and its
approximations were computed on the mesh of $100\times100$ equally spaced
points in $\overline{\mathbf{R}}$. Obtained results together with the graphs
of the integral kernel $\mathbf{K}_{\sech}$ and the approximation error are
given below.

\begin{center}
\noindent%
\begin{tabular}
[c]{|c|c|}\hline
\multicolumn{2}{|c|}{$b=1$}\\\hline
$N$ & Approx. error\\\hline
$1$ & $0.12833$\\
$3$ & $0.021458$\\
$5$ & $0.0017866$\\
$7$ & $8.9155\cdot10^{-5}$\\
$9$ & $2.9655\cdot10^{-6}$\\
$11$ & $7.0469\cdot10^{-8}$\\
$13$ & $1.2562\cdot10^{-9}$\\
$15$ & $7.3683\cdot10^{-11}$\\
$17$ & $7.3991\cdot10^{-11}$\\
$19$ & $7.3989\cdot10^{-11}$\\\hline
\end{tabular}
\qquad%
\begin{tabular}
[c]{|c|c|}\hline
\multicolumn{2}{|c|}{$b=2$}\\\hline
$N$ & Approx. error\\\hline
$5$ & $0.21204$\\
$7$ & $0.043909$\\
$9$ & $0.0059553$\\
$11$ & $0.00057228$\\
$13$ & $4.1076\cdot10^{-5}$\\
$15$ & $2.288\cdot10^{-6}$\\
$17$ & $1.0182\cdot10^{-7}$\\
$19$ & $3.7047\cdot10^{-9}$\\
$21$ & $3.3373\cdot10^{-10}$\\
$23$ & $3.3373\cdot10^{-10}$\\\hline
\end{tabular}
\qquad%
\begin{tabular}
[c]{|c|c|}\hline
\multicolumn{2}{|c|}{$b=4$}\\\hline
$N$ & Approx. error\\\hline
$13$ & $1.8261$\\
$15$ & $0.39987$\\
$17$ & $0.070023$\\
$19$ & $0.010037$\\
$21$ & $0.0011998$\\
$23$ & $0.00012146$\\
$25$ & $1.055\cdot10^{-5}$\\
$27$ & $7.9493\cdot10^{-7}$\\
$29$ & $5.2453\cdot10^{-8}$\\
$31$ & $3.0562\cdot10^{-9}$\\\hline
\end{tabular}
\end{center}

\begin{figure}[tbh]
\centering
\includegraphics[
natheight=2.303800in,
natwidth=3.135700in,
height=2.93in,
width=3.975in
]%
{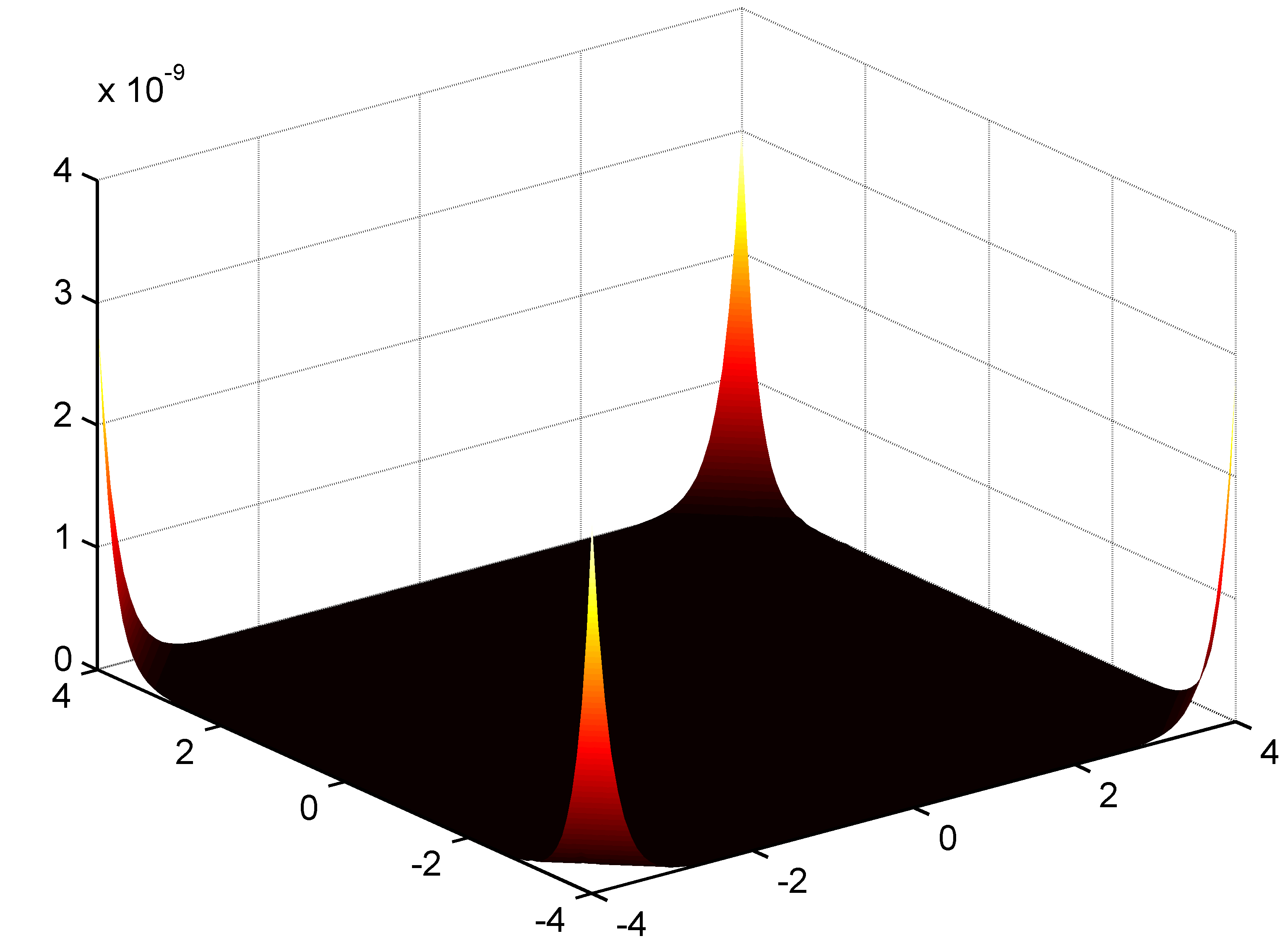}%
\caption{The approximation error of the integral kernel $K_{\sech}$
in the domain $\overline{\mathbf{R}}$  given by $b=4$ by the generalized wave polynomial of order $N=31$.}
\end{figure}

In order to estimate the error of approximation the exact formula
\eqref{Ksech} for the kernel $\mathbf{K}_{\sech}$ was used, where the
corresponding integrals were computed numerically. As can be seen from
\eqref{Ksech} the integrands for $t$ close to $x$ involve subtraction of close
values and division by values close to zero. Due to this fact the kernel
$\mathbf{K}_{\sech}$ computed in this way by \eqref{Ksech} is not exact. This
explains why in the above tables after a certain value of $N$ the reported
approximation error does not decrease, it is limited by the computational
error of $\mathbf{K}_{\sech}$.

To the difference of $\mathbf{K}_{\sech}$ the integral kernel $\mathbf{K}%
_{\cosh}$ does not involve integration and may be computed more precisely. The
approximations of the integral kernel $\mathbf{K}_{\cosh}$ were computed for
$b=2$ and the resulting absolute errors obtained for different numbers of
expansion coefficients are presented in the table below.

\begin{center}
\noindent%
\begin{tabular}
[c]{|c|c||c|c||c|c|}\hline
\multicolumn{6}{|c|}{$b=2$}\\\hline
$N$ & Approx. error & $N$ & Approx. error & $N$ & Approx. error\\\hline
$1$ & $1.7878$ & $11$ & $0.00074903$ & $21$ & $1.316\cdot10^{-10}$\\
$3$ & $1.088$ & $13$ & $5.2024\cdot10^{-5}$ & $23$ & $3.3386\cdot10^{-12}$\\
$5$ & $0.33724$ & $15$ & $2.8243\cdot10^{-6}$ & $25$ & $7.5051\cdot10^{-14}$\\
$7$ & $0.063779$ & $17$ & $1.2312\cdot10^{-7}$ & $27$ & $6.9944\cdot10^{-15}%
$\\
$9$ & $0.0081416$ & $19$ & $4.4042\cdot10^{-9}$ & $29$ & $6.6613\cdot10^{-15}%
$\\\hline
\end{tabular}
\end{center}
\end{example}

The first method of approximate construction of the integral kernels requires
the knowledge of higher order derivatives of the potentials $q_{f}$ and
$q_{1/f}$ in the origin. In the case when such derivatives are unknown or hard
to obtain analytically, instead of performing numerical differentiation we can
approximate the integral kernel by approximating its Goursat data given by
\eqref{GoursatKh2} with truncated series of the form \eqref{+1} and \eqref{-1}.

Indeed, suppose that the numbers $c_{0},\ldots,c_{N}$ and $b_{1},\ldots,b_{N}$
are such that
\[
\biggl|\frac{1}{2}\bigl(\mathbf{K}_{f}(x,x)+\mathbf{K}_{f}(x,-x)\bigr)-c_{0}%
u_{0}(x,x)-\sum_{n=1}^{N}c_{n}u_{2n-1}(x,x)\biggr|<\varepsilon_{1}%
\]
and
\[
\biggl|\frac{1}{2}\bigl(\mathbf{K}_{f}(x,x)-\mathbf{K}_{f}(x,-x)\bigr)-\sum
_{n=1}^{N}b_{n}u_{2n}(x,x)\biggr|<\varepsilon_{2}%
\]
for every $x\in\lbrack-b,b]$. Consider the functions
\begin{equation}
K(x,t)=c_{0}u_{0}(x,t)+\sum_{n=1}^{N}c_{n}u_{2n-1}(x,t)+\sum_{n=1}^{N}%
b_{n}u_{2n}(x,t) \label{Kapprox}%
\end{equation}
and $\widetilde{\mathbf{K}}_{f}=T_{f}^{-1}\mathbf{K}_{f}$ and $\widetilde
{K}=T_{f}^{-1}K$. Then by the definition of the Goursat-to-Goursat
transmutation operator
\[%
\begin{pmatrix}
\widetilde{\mathbf{K}}_{f}(x,x)\\
\widetilde{\mathbf{K}}_{f}(x,-x)
\end{pmatrix}
=T_{G}^{-1}%
\begin{pmatrix}
\mathbf{K}_{f}(x,x)\\
\mathbf{K}_{f}(x,-x)
\end{pmatrix}
\quad\text{and}\quad%
\begin{pmatrix}
\widetilde{K}(x,x)\\
\widetilde{K}(x,-x)
\end{pmatrix}
=T_{G}^{-1}%
\begin{pmatrix}
K(x,x)\\
K(x,-x)
\end{pmatrix}
,
\]
hence due to the boundedness of the operator $T_{G}^{-1}$
\begin{multline*}
\displaybreak[2]
\max\biggl(\max_{x\in\lbrack-b,b]}\bigl|\widetilde{\mathbf{K}}_{f}%
(x,x)-\widetilde{K}(x,x)\bigr|,\max_{x\in\lbrack-b,b]}\bigl|\widetilde
{\mathbf{K}}_{f}(x,-x)-\widetilde{K}(x,-x)\bigr|\biggr)\\
\displaybreak[2]
\leq\Vert T_{G}^{-1}\Vert\max\biggl(\max_{x\in\lbrack-b,b]}\bigl|\mathbf{K}%
_{f}(x,x)-K(x,x)\bigr|,\max_{x\in\lbrack-b,b]}\bigl|\mathbf{K}_{f}%
(x,-x)-K(x,-x)\bigr|\biggr)\\
\displaybreak[2]
\leq\Vert T_{G}^{-1}\Vert\max_{x\in\lbrack-b,b]}\biggl(\biggl|\frac{1}%
{2}\bigl(\mathbf{K}_{f}(x,x)+\mathbf{K}_{f}(x,-x)\bigr)-\frac{1}%
{2}\bigl(K(x,x)+K(x,-x)\bigr)\biggr|\\
+\biggl|\frac{1}{2}\bigl(\mathbf{K}_{f}(x,x)-\mathbf{K}_{f}(x,-x)\bigr)-\frac
{1}{2}\bigl(K(x,x)-K(x,-x)\bigr)\biggr|\biggr)<\Vert T_{G}^{-1}\Vert
(\varepsilon_{1}+\varepsilon_{2}),
\end{multline*}
where we have used equalities $\frac{1}{2}\bigl(K(x,x)+K(x,-x)\bigr)=c_{0}%
u_{0}(x,x)+\sum_{n=1}^{N}c_{n}u_{2n-1}(x,x)$ and $\frac{1}{2}%
\bigl(K(x,x)-K(x,-x)\bigr)=\sum_{n=1}^{N}b_{n}u_{2n}(x,x)$. We obtain from the
proof of \cite[Theorem 3]{KKTT} that for every $(x,t)\in\overline{\mathbf{R}%
}$
\[
\bigl|\widetilde{\mathbf{K}}_{f}(x,t)-\widetilde{K}(x,t)\bigr|\leq3\Vert
T_{G}^{-1}\Vert(\varepsilon_{1}+\varepsilon_{2}),
\]
hence for every $(x,t)\in\overline{\mathbf{R}}$
\[
\bigl|\mathbf{K}_{f}(x,t)-K(x,t)\bigr|\leq3\Vert T_{f}\Vert\cdot\Vert
T_{G}^{-1}\Vert(\varepsilon_{1}+\varepsilon_{2}).
\]
That is, if we approximate the function $g_{1}(x):=\frac{1}{2}\bigl(\mathbf{K}%
_{f}(x,x)+\mathbf{K}_{f}(x,-x)\bigr)=\frac{h}{2}+\frac{1}{4}\int_{0}^{x}%
q_{f}(s)\,ds$ by the functions $\{u_{0}(x,x)\}\cup\{u_{2n-1}(x,x)\}_{n\geq1}$
and the function $g_{2}(x):=\frac{1}{2}\bigl(\mathbf{K}_{f}(x,x)-\mathbf{K}%
_{f}(x,-x)\bigr)=\frac{1}{4}\int_{0}^{x}q_{f}(s)\,ds$ by the functions
$\{u_{2n}(x,x)\}_{n\geq1}$, we obtain an approximation having the form
\eqref{Kapprox} of the integral kernel. It is worth to mention that we do not
need to know the Darboux associated potential $q_{1/f}$ for this method and we
do not impose additional smoothness conditions on the potential $q_{f}$. The
approximation of the functions $g_{1}$ and $g_{2}$ by the corresponding
combinations of the functions $u_{n}(x,x)$ can be done in several ways. We may
use the least squares method to obtain a good, however far from the best,
approximation. We may reformulate the approximation problem as a linear
programming problem and solve it. Under the additional assumption that the
corresponding systems of the functions $u_{n}(x,x)$ satisfy the Haar
condition, we may use the Remez algorithm. We refer the reader to
\cite[Section 6]{KKTT} and to \cite[\S 7]{Meinardus}, \cite[Chapter 6]{Rice}
and \cite[Chapter 1]{Rivlin} for further details.

\begin{example}
We approximated the integral kernels $\mathbf{K}_{\sech}$ and $\mathbf{K}%
_{\cosh}$ from Example \ref{ExampleCosh2} by the second described method. All
the parameters are taken as in the previous example. The Goursat data were
approximated by corresponding functions $u_{n}(x,x)$ with the help of the
Remez simple exchange algorithm, see e.g., \cite[\S 7]{Meinardus}.%

\begin{figure}[tbh]
\centering
\includegraphics[
natheight=2.342900in,
natwidth=3.135700in,
height=2.93in,
width=3.975in
]%
{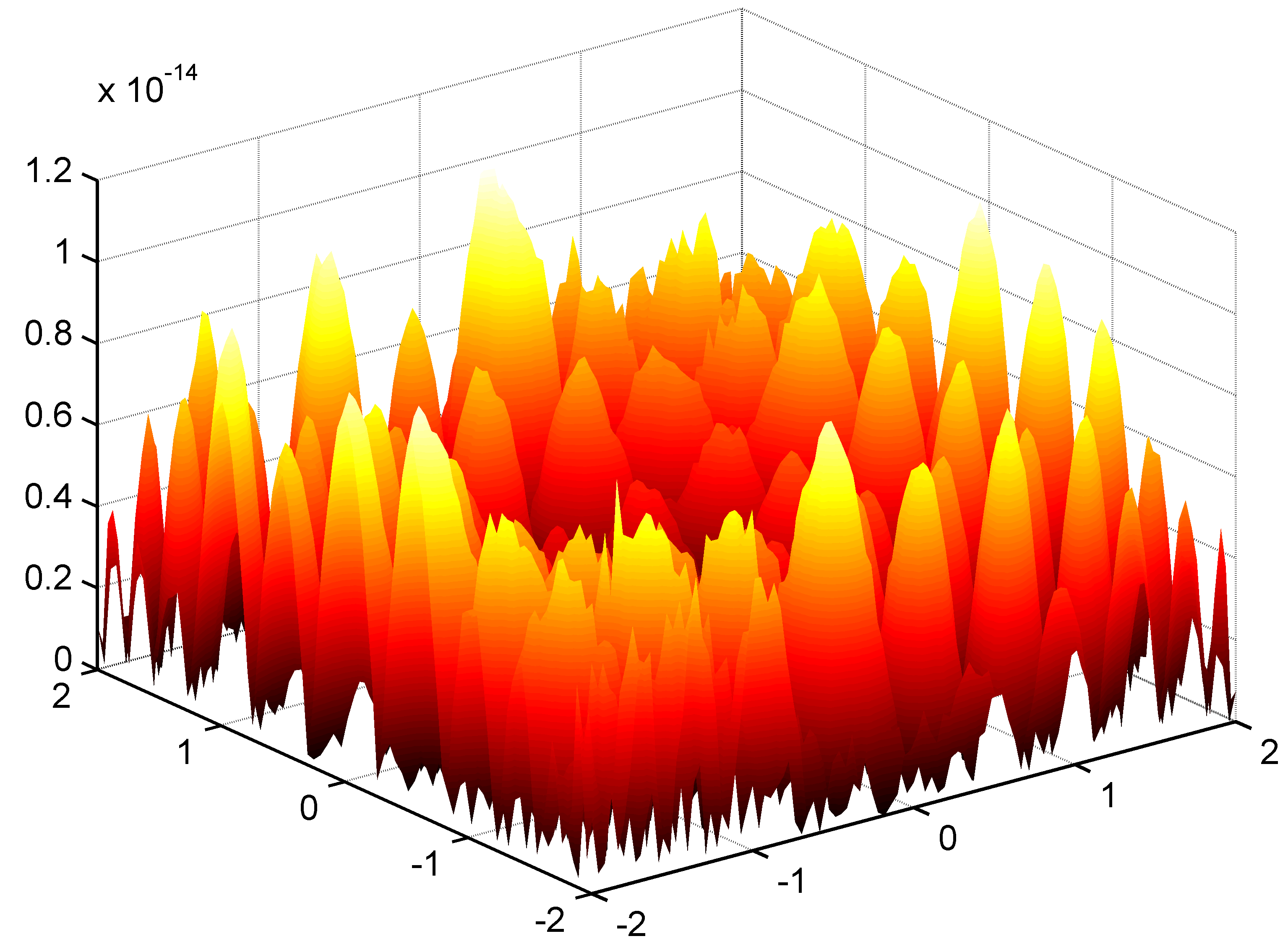}%
\caption{The approximation error of the integral kernel $K_{\cosh}$
in the domain $\overline{\mathbf{R}}$ given by
$b=2$ by the generalized wave polynomial of order $N=19$. The
generalized wave polynomial is obtained with the use of the Remez simple
exchange algorithm.}
\end{figure}

\begin{center}
\noindent%
\begin{tabular}
[c]{|c|c|}\hline
\multicolumn{2}{|c|}{$b=2$, $\mathbf{K}_{\sech}$ kernel}\\\hline
$N$ & Approx. error\\\hline
$5$ & $0.0045993$\\
$9$ & $9.3687\cdot10^{-6}$\\
$13$ & $4.1549\cdot10^{-9}$\\
$17$ & $3.3354\cdot10^{-10}$\\\hline
\end{tabular}
\qquad%
\begin{tabular}
[c]{|c|c|}\hline
\multicolumn{2}{|c|}{$b=4$, $\mathbf{K}_{\sech}$ kernel}\\\hline
$N$ & Approx. error\\\hline
$13$ & $7.4042\cdot10^{-5}$\\
$17$ & $2.342\cdot10^{-7}$\\
$21$ & $2.789\cdot10^{-10}$\\
$25$ & $4.9467\cdot10^{-11}$\\\hline
\end{tabular}
\qquad%
\begin{tabular}
[c]{|c|c|}\hline
\multicolumn{2}{|c|}{$b=2$, $\mathbf{K}_{\cosh}$ kernel}\\\hline
$N$ & Approx. error\\\hline
$5$ & $0.0052907$\\
$9$ & $1.2563\cdot10^{-5}$\\
$13$ & $6.6227\cdot10^{-9}$\\
$17$ & $1.1813\cdot10^{-12}$\\
$19$ & $1.0325\cdot10^{-14}$\\\hline
\end{tabular}
\end{center}
\end{example}

\section{Applications to Sturm-Liouville spectral problems}\label{SectASL}
Solution of Sturm-Liouville spectral problems is a straightforward application of the developed theory. In this section we briefly explain the convenience of the approximation \eqref{Kapprox} of the integral kernel in the simplest setting of Dirichlet boundary conditions. The separate paper \cite{KrT2013} contains further details.

Consider the spectral problem for the equation
\begin{equation}\label{EqSLproblem}
    -u''+qu=\omega^2 u
\end{equation}
on the interval $[0,b]$ with the boundary conditions
\begin{equation}\label{EqSLBC}
    u(0)=u(b)=0.
\end{equation}
Let $\mathbf{T}_h$ be a transmutation operator for \eqref{EqSLproblem} where we assume that the potential $q$ is extended onto $[-b,b]$. From \eqref{mapping IC} we have that the function $s(x;\omega):=\mathbb{T}_h(\sin\omega x)$ is a solution of \eqref{EqSLproblem} satisfying the first of the conditions \eqref{EqSLBC}. Thus, the characteristic equation of the problem \eqref{EqSLproblem}, \eqref{EqSLBC} has the form
\begin{equation*}
    \sin\omega b +\int_{-b}^{b}\mathbf{K}(b,t;h)\sin\omega t\,dt=0.
\end{equation*}

Let $f$ be a non-vanishing solution of the equation $-u''+qu=0$ on $[-b,b]$ normalized as $f(0)=1$ and $h:=f'(0)$. Consider an approximation $K_N(x,t)$ of the integral kernel $\mathbf{K}(x,t;h)$ of the form \eqref{Kapprox} and define $s_N(x;\omega):=\sin\omega x+\int_{-x}^x K_N(x,t)\sin\omega t\,dt$. By the definition of the generalized wave polynomials \eqref{um} and the parity relations \eqref{umParity} we have
\begin{equation}\label{sN}
    s(x;\omega) \cong s_{N}(x;\omega) = \sin \omega x + \sum_{n=1}^{N}b_{n}\sum_{\text{odd }k=1}^{n}\binom{n}{k}\varphi
_{n-k}(x)\int_{-x}^{x}t^{k}\sin \omega t\,dt.
\end{equation}
The integrals in \eqref{sN} can be easily evaluated explicitly, hence \eqref{sN} presents a convenient approximation of the solution $s(x;\omega)$. Moreover, the error of the approximation can be uniformly bounded for all real $\omega$. If $|\mathbf{K}(x,t;h)-K_N(x,t)|\le\varepsilon$ in $0\le|t|\le|x|\le b$ then
\begin{equation*}
    |s(x;\omega)-s_N(x;\omega)|\le\int_{-x}^{x}|\mathbf{K}(x,t;h)-K_N(x,t)|\cdot|\sin\omega t|\,dt\le \varepsilon \int_{-x}^{x}|\sin\omega t|\,dt\le 2\varepsilon x.
\end{equation*}

Thus, the proposed method for solving \eqref{EqSLproblem}, \eqref{EqSLBC} consists of the following steps: 1) compute a non-vanishing solution $f$ of $-u''+qu=0$; 2) compute $N+1$ functions $\varphi_k$; 3) find coefficients $b_n$, $n=1,\ldots,N$ of the approximation of the integral kernel $\mathbf{K}(x,t;h)$; 4) find zeros of $\left. s_N(x,\omega)\right|_{x=b}$.

\begin{example}
Consider the spectral problem \eqref{EqSLproblem}, \eqref{EqSLBC} with $q(x)=e^x$ and $b=\pi$ (see \cite[Appendix A]{Pryce}). The exact characteristic equation of the problem is given by $I_{2 i \omega }(2)I_{-2 i \omega }\big(2
\sqrt{e^\pi}\big)-I_{-2 i \omega }(2) I_{2 i
\omega }\big(2 \sqrt{e^\pi}\big)=0$, allowing one to calculate an arbitrary number of eigenvalues for comparison with numerical results.

We choose $f(x)=I_0\big(2e^{x/2}\big)/I_0(2)$ as a non-vanishing particular solution and performed approximation of the integral kernel by approximating its Goursat data as described in Section \ref{SectApproxKernel} using $N=30$. In the following table we present approximations of the first 1000 eigenvalues obtained by the proposed method together with corresponding absolute and relative errors. Observe that the absolute errors remain essentially of the same order and relative errors approach the machine precision bound. The computation time is within seconds.
\begin{center}
\noindent%
\begin{tabular}{|c|c|c|c|}\hline
$n$ & Obtained $\omega_n^2$ & Abs. error & Rel. error\\\hline
$1$ & 4.89666937996891 & $1.2\cdot10^{-12}$ & $2.5\cdot10^{-13}$\\
$2$ & 10.0451898932577 & $4.0\cdot10^{-12}$& $4.0\cdot10^{-13}$\\
$3$ & 16.0192672505157 & $2.3\cdot10^{-11}$& $1.5\cdot10^{-12}$\\
$5$ & 32.2637070458132 & $8.7\cdot10^{-12}$& $2.7\cdot10^{-13}$\\
$10$ & 107.116676138236 & $3.2\cdot10^{-11}$& $3.0\cdot10^{-13}$\\
$20$ & 407.065235267218 & $1.2\cdot10^{-10}$& $3.0\cdot10^{-13}$\\
$50$ & 2507.05043440902 &$1.2\cdot10^{-10}$& $4.9\cdot10^{-14}$\\
$100$ & 10007.0483099952 &$2.4\cdot10^{-11}$& $2.4\cdot10^{-15}$\\
$200$ & 40007.0477785361 &$3.6\cdot10^{-11}$& $9.1\cdot10^{-16}$\\
$500$ & 250007.047629702 &$1.2\cdot10^{-10}$& $4.7\cdot10^{-16}$\\
$1000$ & 1000007.04760844 &$2.3\cdot10^{-10}$& $2.3\cdot10^{-16}$\\
\hline
\end{tabular}
\end{center}
\end{example}

\affiliationone{
   Vladislav V. Kravchenko \\
   Department of Mathematics,\\
   CINVESTAV del IPN,\\
   Unidad Queretaro,\\
   Libramiento Norponiente No. 2000, \\
   Fracc. Real de Juriquilla, Queretaro,\\
   Qro. C.P. 76230 MEXICO
   \email{vkravchenko@math.cinvestav.edu.mx}
}%
\affiliationtwo{
   Sergii M. Torba\\
   Department of Mathematics,\\
   CINVESTAV del IPN,\\
   Unidad Queretaro,\\
   Libramiento Norponiente No. 2000, \\
   Fracc. Real de Juriquilla, Queretaro,\\
   Qro. C.P. 76230 MEXICO
   \email{storba@math.cinvestav.edu.mx}
}%
\end{document}